\documentclass[11pt,twoside]{article}

\usepackage{mathrsfs,amsfonts,amsmath,amssymb}
\usepackage{latexsym}
\newtheorem{satz}{Theorem}
\newtheorem{proposition}[satz]{Proposition}
\newtheorem{theorem}[satz]{Theorem}
\newtheorem{lemma}[satz]{Lemma}
\newtheorem{definition}[satz]{Definition}
\newtheorem{corollary}[satz]{Corollary}
\newtheorem{remark}[satz]{Remark}
\newtheorem{example}[satz]{Example}

\def\eps{\varepsilon}
\def\_phi{\varphi}

\def\a{\alpha}

\def\la{\lambda}

\def\v{\vec}
\def\F{{\mathbb F}}
\def\L{\Lambda}
\def\m{\times}
\def\t{\tilde}
\def\o{\omega}
\def\ov{\overline}

\def\C{{\mathbb C}}
\def\R{{\mathbb R}}
\def\E{\mathsf {E}}
\def\T{{\mathbb T}}

\def\Z_N{{\mathbb Z}_N}
\def\Z{{\mathbb Z}}
\def\N{{\mathbb N}}
\def\U{{\mathcal U}}

\def\f{{\mathbb F}}
\def\Gr{{\mathbf G}}

\def\D{{\mathbb D}}

\def\FF{\widehat}
\def\c{\circ}
\def\D{\Delta}
\def\Cf{{\mathcal C}}

\def\T{\mathsf {T}}

\author{Shkredov I.D.}
\title{ Energies and structure of additive sets
\footnote{
This work was supported by grant
mol\underline{ }a\underline{ }ved 12--01--33080.}
}
\date{}
\begin{document}
\maketitle

\begin{center}
 Annotation.
\end{center}

{\it \small
    In the paper
    we prove that any sumset or difference set has large $\E_3$ energy.
    Also, we give a full description of  families of  sets having
    critical relations between some kind of energies such as $\E_k$, $\T_k$ and Gowers
    norms.
    In particular, we give
    criteria
    for a set to be a \\
$\bullet~$ set of the form $H\dotplus \L$, where $H+H$ is small and $\L$ has "random structure", \\
$\bullet~$ set equals a disjoint union of sets $H_j$, each $H_j$ has small doubling,\\
$\bullet~$ set having large subset $A'$ with $2A'$ is equal to a set with small doubling
and $|A'+A'| \approx |A|^4 / \E(A)$.
}
\\

\section{Introduction}
\label{sec:introduction}

Let $\Gr = (\Gr,+)$ be an abelian group.
For two sets $A,B\subseteq \Gr$ define the {\it sumset} as
$$A+B := \{ x\in \Gr ~:~ x=a+b \,,a\in A\,,b\in B \}$$
and, similarly, the {\it difference} set
$$A-B := \{ x\in \Gr ~:~ x=a-b \,,a\in A\,,b\in B \} \,.$$
Also denote the {\it additive energy} of a set $A$ by
$$
    \E(A) = \E_2 (A) = |\{ a_1-a'_1 = a_2-a'_2 ~:~ a_1,a'_1,a_2,a'_2 \in A \}| \,,
$$
and
{\it $\E_k (A)$ energy} as
\begin{equation}\label{f:start_E_k}
    \E_k (A) = |\{ a_1-a'_1 = a_2-a'_2 = \dots = a_k-a'_k ~:~ a_1,a'_1,\dots, a_k,a'_k \in A \}| \,.
\end{equation}
The special case $k=1$
gives us $\E_1 (A) = |A|^2$
because of there is no any restriction in the set from (\ref{f:start_E_k}).
So, the cardinality of a set can be considered as a degenerate sort of energy.
Note that a trivial upper bound for $\E_k (A)$ is $|A|^{k+1}$.
Now recall a well--known Balog--Szemer\'{e}di--Gowers Theorem \cite{TV}.

\begin{theorem}
    Let $A\subseteq \Gr$ be a set, and $K\ge 1$ be a real number.
    Suppose that $\E(A) \ge |A|^3 / K$.
    Then there is $A' \subseteq A$ such that $|A'| \gg |A|/K^C$ and
\begin{equation}\label{f:BSzG_introduction}
     |A'-A'| \ll K^C |A'| \,,
\end{equation}
    where $C>0$ is an absolute constant.
\label{t:BSzG_introduction}
\end{theorem}

So, Balog--Szemer\'{e}di--Gowers Theorem can be considered as a result
about
the structure
of sets $A$ having the extremal (in terms of its cardinality or $\E_1(A)$ in other words) value of $\E(A)$.
Namely,
any of such a set has a subset $A'$ with the extremal value
of the cardinality of its difference set.
These sets $A'$ are called sets with {\it small doubling}.
On the other hand, it is easy to obtain, using the Cauchy--Schwarz inequality,
that any $A$ having subset $A'$
such that
(\ref{f:BSzG_introduction}) holds, automatically has polynomially large energy
$\E(A) \gg_K |A|^3$ (see e.g. \cite{TV}).
Moreover, the structure of sets with small doubling is known more or less thanks to a well--known
Freiman's theorem (see \cite{TV} or \cite{Sanders_2A-2A}).
Thus,
Theorem \ref{t:BSzG_introduction}
finds subsets of $A$
with rather rigid structure
and,
actually,
it is a criterium for a set $A$ to be a set with large (in terms of $\E_1(A)$) the additive energy:
$\E(A) \sim_K |A|^3 \sim_K (\E_1 (A))^{3/2}$.

In the paper we consider another extremal relations between different energies and describe the structure
of sets having these critical relations.
Such kind of theorems have plenty of applications.
It is obvious for Balog--Szemer\'{e}di--Gowers theorem, see
e.g. \cite{TV}, \cite{BK_AP3}, \cite{Gow_4}, \cite{Gow_m}, \cite{KSh}, \cite{M_R-N_S} and so on;
for recent
applications using
critical relations between energies  $\E_2 (A)$ and $\E_3 (A)$,
see e.g. \cite{s_ineq}, \cite{s_mixed} and others.
Before formulate our main results let us recall a beautiful theorem of Bateman--Katz
\cite{BK_AP3}, \cite{BK_struct}
which is another example of theorems are called "structural"\, by us.

\begin{theorem}
Let $A \subseteq \Gr$ be a symmetric set, $\tau_0$, $\sigma_0$ be nonnegative real numbers
and $A$ has the property that for any $A_*\subseteq A$, $|A_*| \gg |A|$ the following holds
$\E(A_*) \gg \E (A) = |A|^{2+\tau_0}$.
Suppose that $\T_4 (A) \ll |A|^{4+3\tau_0+\sigma_0}$.
Then there exists a function $f_{\tau_0} : (0,1) \to (0,\infty)$ with $f_{\tau_0} (\eta) \to 0$ as $\eta \to 0$
and a number $0\le \a \le \frac{1-\tau_0}{2}$ such that there are sets $X_j,H_j\subseteq \Gr$, $B_j\subseteq A$,
$j\in [|A|^{\a-f_{\tau_0} (\sigma_0)}]$ with
\begin{equation}\label{f:BK_1}
    |H_j| \ll |A|^{\tau_0+\a+ f_{\tau_0} (\sigma_0)} \,,\quad \quad |X_j| \ll |A|^{1-\tau_0-2\a+ f_{\tau_0} (\sigma_0)} \,,
\end{equation}
\begin{equation}\label{f:BK_2}
    |H_j-H_j| \ll |H_j|^{1+f_{\tau_0} (\sigma_0)} \,,
\end{equation}
\begin{equation}\label{f:BK_3}
    |(X_j+H_j) \cap B_j| \gg |A|^{1-\a-f_{\tau_0} (\sigma_0)} \,,
\end{equation}
and $B_i \cap B_j = \emptyset$ for all $i\neq j$.
\label{t:BK_structural}
\end{theorem}

Here $\T_4 (A)$ is the number of solutions of the equation $a_1+a_2+a_3+a_4=a'_1+a'_2+a'_3+a'_4$,
$a_1,a_2,a_3,a_4,a'_1,a'_2,a'_3,a'_4 \in A$
and this characteristic is another sort of energy.
One can check that any set satisfying (\ref{f:BK_1})---(\ref{f:BK_3}), $\E(A) = |A|^{2+\tau_0}$
and all another conditions of the theorem is an example of a set having $\T_4 (A) \ll |A|^{4+3\tau_0+\sigma_0}$.
Note that if $\E (A) = |A|^{2+\tau_0}$ then by the H\"{o}lder inequality one has $\T_4 (A) \ge |A|^{4+3\tau_0}$.
Thus, Theorem \ref{t:BK_structural} gives us a full description of sets having critical relations between
a pair of two energies:
$\E(A)$ and $\T_4 (A)$.

There are two opposite extremal cases in Theorem \ref{t:BK_structural} : $\a = 0$ and $\a = \frac{1-\tau_0}{2}$.
For simplicity consider the situation when $\Gr = \f_2^n$.
In the case $\a = 0$ by Bateman--Katz result  our set $A$,  roughly speaking, is close to a set of the form
$H \dotplus \Lambda$, where $\dotplus$ means the direct sum,
$H\subseteq \mathbf{F}_2^n$ is a subspace, and $\Lambda \subseteq \mathbf{F}_2^n$ is a dissociated set (basis),
$|\L| \sim |A| /|H| \sim |A|^{1-\tau_0}$.
These sets are interesting in its own right being counterexamples in many problems of additive combinatorics.
The reason for this is that they have mixed properties : on the one hand they contain translations $H+\la$, $\la \in \L$
of really structured set $H$ but on the other hand they have also some random properties, for example, its Fourier
coefficients (see the definition in section \ref{sec:definitions}) are small.
Our first result says that a set $A$ is close to a set of the form $H \dotplus \Lambda$ iff there is
the critical relation between $\E_3 (A)$ and $\E(A)$, that is $\E_3(A) \gg |A|\E(A)$, more precisely,
see Theorem \ref{t:H+L_description}.

In the situation $\a = \frac{1-\tau_0}{2}$, $\Gr = \f_2^n$ our set $A$ looks like a union
of (additively) disjoint subspaces $H_1,\dots,H_k$ (see example (iii) from \cite{Sanders_survey2})
with $k = |A|^{\frac{1+\tau_0}{2}}$.
Such sets
can be
called {\it self--dual} sets, see \cite{s_mixed}.
In our second result we show that, roughly speaking,
any such a set has critical relation between $\E_3 (A)$
(more precisely
$\E(A) \cdot \E_4(A)$)
and so--called
Gowers $U^3$--norm of the set $A$ (see the definition in section \ref{sec:Gowers}) and vice versa.
Theorem \ref{t:self-dual} contains the
exact
formulation.

These two structural results on sets having critical relations between a pair of its energies are the hearth of our paper.
In the opposite of Theorem \ref{t:BK_structural}
almost
all bounds of the paper are polynomial, excluding, of course,
the dependence on the number $k$ of the considered energies
$\E_k$, $\T_k$ or $U^k$ if its appear.
Moreover
the first
structural theorem
hints us a partial  answer to the following
important  question.
Consider the difference set $D=A-A$ or the sumset $S=A+A$ of an arbitrary set $A$.
What can we say nontrivial about the energies of $D$, $S$ in terms of the energies of $A$?
In view of the first of the main examples
above,
that is $\Gr = \f_2^n$, $A=H\dotplus \L$,
$|\L| = K$, $\E(A) \sim |A|^3 / K$
we cannot hope to obtain a nontrivial bound for the additive energy of $D$ or $S$ because in the case
$D=S=H \dotplus (\L+\L)$, and so it has a similar structure to $A$ with $\L$ replacing by $\L+\L$.
On the other hand, we know that the sets of the form $H\dotplus \L$ have large $\E_3$ energy.
Thus, one can hope to
obtain
a good lower bound for $\E_3 (D)$ and $\E_3 (S)$.
It turns out to be the case and we prove it in section \ref{sec:sumsets1}.
Roughly speaking, our result asserts that if $|D|=K|A|$, $\E(A) \ll |A|^3 /K$ then
\begin{equation}\label{f:start_E3D}
    \E_3 (D) \gg K^{7/4} |A|^4 \,,
\end{equation}
and a similar inequality for $A+A$.

The paper is organized as follows.
We start  with definitions and notations used in the paper.
In the next section we give several characterisations of sets of the form $A=H\dotplus \L$, where
$H$ is a set with small doubling, $\L$ is a "dissociated"\, set.
Also we consider a "dual"\, question on sets having critical relations between $\T_4$ and $\E$ energies,
that is the situation when $\T_4 (A)$ is large in terms on $\E(A)$.
It was proved that, roughly, $A$ contains a large subset $A'$ such that the sequence $A', 2A', 3A', \dots$
is stabilized at the second step, namely, $A'+A'$ is a set with small doubling
and, besides, $|A'+A'| \approx |A|^4 / \E (A)$
in the
only case when
$\T_4 (A) \gg |A|^2 \E(A)$, see Theorem \ref{t:T_3_and_E_critical}.

Section \ref{sec:sumsets1} contains the proof of inequality (\ref{f:start_E3D})
and we make some preliminaries to this in section \ref{sec:sumsets}.
For example, we obtain in the section an interesting characterisation of sumsets $S=A+A$ or difference sets $D=A-A$ with
extremal cardinalities of intersections
$$|A| \le |D\cap (D+x_1) \cap \dots \cap (D+x_s)| \le |A|^{1+o(1)} \,,$$
and, similarly, for $S$,
see Theorem \ref{t:dichotomy_DS}.
It turns out that for such sets $D$, $S$ the set $A$
should have either very small $O(|A|^{k+o(1)})$
the energy $\E_k (A)$
or very large $\gg |A|^{3-o(1)}$
the additive energy.
In other words either $A$ has "random behaviour"\,
or, in contrary, is very structured.
Clearly, both situations are realized: the first one in the situation when
$A$ is a fair random set (and hence $A\pm A$ has almost no structure) and the second one
if $A$ is a set with small doubling, say.

In section \ref{sec:Gowers} we consider  some simple properties of Gowers norms of
{\it the characteristic function of a set $A$} and
prove a preliminary result on the connection of $\E(A)$ with $\E(A\cap (A+s))$, $s\in A-A$,
see Theorem \ref{t:E(A_s)}.
It gives a partial counterexample to a famous construction of Gowers \cite{Gow_4}, \cite{Gow_m} of uniform sets
with non-uniform intersections $\E(A\cap (A+s))$
(see the definitions in \cite{Gow_4}, \cite{Gow_m} or \cite{TV}).
We show that although all sets $A\cap (A+s)$ can be non-uniform but there is always $s\neq 0$
such that $\E(A\cap (A+s)) \ll |A\cap (A+s)|^{3-c}$, $c>0$, provided by some weak conditions take place.
This question was asked to the author by T. Schoen.

In the next section we develop the investigation from the previous one and
characterize all sets with critical relation between Gowers $U^3$--norm and the energies  $\E,\E_4$ or $\E_3$.
Also we consider some questions on finding in $A$ a family of disjoint sets $A\cap (A+s)$ or its large disjoint subsets.

A lot of results of the paper such as Bateman--Katz theorem  are
proved under
some regular conditions
on $A$.
For example, the
assumption
from Theorem \ref{t:BK_structural} require that
for  all $A_*\subseteq A$, $|A_*| \gg |A|$ the following holds $\E(A_*) \gg \E (A)$.
We call the conditions as connectedness of our set $A$
(see the definitions from sections \ref{sec:preliminaries}, \ref{sec:Gowers})
and prove in the appendix that any set contains some large connected subset.
Basically, we
generalize the method from \cite{s_doubling}.

Thus, we have characterized two extremal
situations
of Theorem \ref{t:BK_structural} in terms of energies.
Is there some similar characterisation for other cases of the result?
Do exist criteria in terms of
energies for another families of sets?
Finally,
are there
further
characteristics of sumsets/difference sets which separate it from arbitrary sets?

The author is grateful to Vsevolod F. Lev and Tomasz Schoen
for useful discussions.

\section{Definitions}
\label{sec:definitions}

Let $\Gr$ be an abelian group.
If $\Gr$ is finite then denote by $N$ the cardinality of $\Gr$.
It is well--known~\cite{Rudin_book} that the dual group $\FF{\Gr}$ is isomorphic to $\Gr$ in the case.
Let $f$ be a function from $\Gr$ to $\mathbb{C}.$  We denote the Fourier transform of $f$ by~$\FF{f},$
\begin{equation}\label{F:Fourier}
  \FF{f}(\xi) =  \sum_{x \in \Gr} f(x) e( -\xi \cdot x) \,,
\end{equation}
where $e(x) = e^{2\pi i x}$
and $\xi$ is a homomorphism from $\FF{\Gr}$ to $\R/\Z$ acting as $\xi : x \to \xi \cdot x$.
We rely on the following basic identities
\begin{equation}\label{F_Par}
    \sum_{x\in \Gr} |f(x)|^2
        =
            \frac{1}{N} \sum_{\xi \in \FF{\Gr}} \big|\widehat{f} (\xi)\big|^2 \,,
\end{equation}
\begin{equation}\label{svertka}
    \sum_{y\in \Gr} \Big|\sum_{x\in \Gr} f(x) g(y-x) \Big|^2
        = \frac{1}{N} \sum_{\xi \in \FF{\Gr}} \big|\widehat{f} (\xi)\big|^2 \big|\widehat{g} (\xi)\big|^2 \,,
\end{equation}
and
\begin{equation}\label{f:inverse}
    f(x) = \frac{1}{N} \sum_{\xi \in \FF{\Gr}} \FF{f}(\xi) e(\xi \cdot x) \,.
\end{equation}
If
$$
    (f*g) (x) := \sum_{y\in \Gr} f(y) g(x-y) \quad \mbox{ and } \quad
        (f\circ g) (x) := \sum_{y\in \Gr} f(y) g(y+x)
$$
 then
\begin{equation}\label{f:F_svertka}
    \FF{f*g} = \FF{f} \FF{g} \quad \mbox{ and } \quad \FF{f \circ g} = \FF{f^c} \FF{g} = \ov{\FF{\ov{f}}} \FF{g} \,,
\end{equation}
where for a function $f:\Gr \to \mathbb{C}$ we put $f^c (x):= f(-x)$.
 Clearly,  $(f*g) (x) = (g*f) (x)$ and $(f\c g)(x) = (g \c f) (-x)$, $x\in \Gr$.
 The $k$--fold convolution, $k\in \N$  we denote by $*_k$,
 so $*_k := *(*_{k-1})$.

We use in the paper  the same letter to denote a set
$S\subseteq \Gr$ and its characteristic function $S:\Gr\rightarrow \{0,1\}.$
Clearly, $S$ is the characteristic function of  a set iff
\begin{equation}\label{f:char_char}
    \FF{S} (x) = N^{-1} (\ov{\FF{S}} \c \FF{S}) (x) \,.
\end{equation}
Write $\E(A,B)$ for the {\it additive energy} of two sets $A,B \subseteq \Gr$
(see e.g. \cite{TV}), that is
$$
    \E(A,B) = |\{ a_1+b_1 = a_2+b_2 ~:~ a_1,a_2 \in A,\, b_1,b_2 \in B \}| \,.
$$
If $A=B$ we simply write $\E(A)$ instead of $\E(A,A).$
Clearly,
\begin{equation}\label{f:energy_convolution}
    \E(A,B) = \sum_x (A*B) (x)^2 = \sum_x (A \circ B) (x)^2 = \sum_x (A \circ A) (x) (B \circ B) (x)
    \,.
\end{equation}
and by (\ref{svertka}),
\begin{equation}\label{f:energy_Fourier}
    \E(A,B) = \frac{1}{N} \sum_{\xi} |\FF{A} (\xi)|^2 |\FF{B} (\xi)|^2 \,.
\end{equation}
Let
$$
   \T_k (A) := | \{ a_1 + \dots + a_k = a'_1 + \dots + a'_k  ~:~ a_1, \dots, a_k, a'_1,\dots,a'_k \in A \} |
    =
        \frac{1}{N} \sum_{\xi} |\FF{A} (\xi)|^{2k}
$$
and more generally
$$
    \T_k (A_1,\dots,A_k) :=
        | \{ a_1 + \dots + a_k = a'_1 + \dots + a'_k  ~:~ a_1,a'_1 \in A_1, \dots, a_k, a'_k \in A_k \} | \,.
$$
Let also
$$
    \sigma_k (A) := (A*_k A)(0)=| \{ a_1 + \dots + a_k = 0 ~:~ a_1, \dots, a_k \in A \} | \,.
$$
Notice that for a symmetric set $A$ that is $A=-A$ one has $\sigma_2
(A) = |A|$ and $\sigma_{2k} (A) = \T_k (A)$.
Having a set $P\subseteq A-A$ we write $\sigma_P (A) := \sum_{x\in P} (A\c A) (x)$.

 For a sequence $s=(s_1,\dots, s_{k-1})$ put
$A^B_s = B \cap (A-s_1)\dots \cap (A-s_{k-1}).$
If $B=A$ then write $A_s$ for $A^A_s$.
Let
\begin{equation}\label{f:E_k_preliminalies}
    \E_k(A)=\sum_{x\in \Gr} (A\c A)(x)^k = \sum_{s_1,\dots,s_{k-1} \in \Gr} |A_s|^2
\end{equation}
and
\begin{equation}\label{f:E_k_preliminalies_B}
\E_k(A,B)=\sum_{x\in \Gr} (A\c A)(x) (B\c B)(x)^{k-1} = \sum_{s_1,\dots,s_{k-1} \in \Gr} |B^A_s|^2
\end{equation}
be the higher energies of $A$ and $B$.
The second formulas in (\ref{f:E_k_preliminalies}), (\ref{f:E_k_preliminalies_B})
can be considered as the definitions of $\E_k(A)$, $\E_k(A,B)$ for non integer $k$, $k\ge 1$.
Similarly, we write $\E_k(f,g)$ for any complex functions $f$, $g$
and more generally
$$
    \E_k (f_1,\dots,f_{k}) = \sum_x (f_1 \c f_1) (x) \dots (f_k \c f_k) (x) \,.
$$
Putting $\E_1 (A) = |A|^2$.
For a set $P\subseteq \Gr$ write $\E^P_k (A) := \sum_{s\in P} |A_s|^k$, $\E^P (A) := \E^P_2 (A)$.
We
put
$\E^*_k (A)$ for $\E^*_k (A) = \sum_{s\neq 0} |A_s|^k$.

Clearly,
\begin{eqnarray}\label{f:energy-B^k-Delta}
\E_{k+1}(A,
B)&=&\sum_x(A\c A)(x)(B\c B)(x)^{k}\nonumber \\
&=&\sum_{x_1,\dots, x_{k-1}}\Big (\sum_y A(y)B(y+x_1)\dots
B(y+x_{k})\Big )^2 =\E(\Delta_k (A),B^{k}) \,,
 \end{eqnarray}
where
$$
    \Delta (A) = \Delta_k (A) := \{ (a,a, \dots, a)\in A^k \}\,.
$$
We also put $\Delta(x) = \Delta (\{ x \})$, $x\in \Gr$.


\bigskip

Quantities $\E_k (A,B)$ can be written in terms of generalized convolutions.

\begin{definition}
   Let $k\ge 2$ be a positive number, and $f_0,\dots,f_{k-1} : \Gr \to \C$ be functions.
Denote by
$${\mathcal C}_k (f_0,\dots,f_{k-1}) (x_1,\dots, x_{k-1})$$
the function
$$
    \Cf_k (f_0,\dots,f_{k-1}) (x_1,\dots, x_{k-1}) = \sum_z f_0 (z) f_1 (z+x_1) \dots f_{k-1} (z+x_{k-1}) \,.
$$
Thus, $\Cf_2 (f_1,f_2) (x) = (f_1 \circ f_2) (x)$.
If $f_1=\dots=f_k=f$ then write
$\Cf_k (f) (x_1,\dots, x_{k-1})$ for $\Cf_k (f_1,\dots,f_{k}) (x_1,\dots, x_{k-1})$.
\end{definition}

In particular, $(\Delta_k (B) \c A^k) (x_1,\dots,x_k) = \Cf_{k+1} (B,A,\dots,A) (x_1,\dots,x_k)$, $k\ge 1$.

Quantities $\E_k (A)$ and $\T_k (A)$ are "dual"\, in some sense.
For example in \cite{s_mixed}, Note 6.6 (see also \cite{SS1}) it was proved that
$$
    \left( \frac{\E_{3/2} (A)}{|A|} \right)^{2k}
        \le
            \E_k (A) \T_k (A) \,,
$$
provided by $k$ is even.
Moreover, from (\ref{F:Fourier})---(\ref{f:inverse}), (\ref{f:char_char}) it follows that
$\t{\E}_{2k} (\FF{A}) := \sum_{x} ( \ov{\FF{A}} \c \FF{A})^{k} (x) (\FF{A} \c \ov{\FF{A}})^{k} (x)
= N^{2k+1} \T_k (A)$ and $\T_k (|\FF{A}|^2) = N^{2k-1} \E_{2k} (A)$.

\bigskip

For a positive integer $n,$ we set $[n]=\{1,\ldots,n\}.$
Let $x$ be a vector. By $\| x \|$ denote the number of components of $x$.
All logarithms are to base $2.$ Signs $\ll$ and $\gg$ are the usual Vinogradov's symbols
and if the bounds depend on some parameter $M$ {\it polynomially} then we write $\ll_M$, $\gg_M$.
If for two numbers $a$, $b$ the following holds $a \ll_M b$, $b \ll_M a$ then we write $a \sim_M b$.
In particular, $a \sim b$ means $a\ll b$ and $b\ll a$.

All polynomial bounds in the paper can be obtained in explicit way.


\newpage

\section{Preliminaries}
\label{sec:preliminaries}

Let us begin with the famous Pl\"{u}nnecke--Ruzsa inequality (see  \cite{petridis} or \cite{TV}, e.g.).

\begin{lemma}
Let $A\subseteq \Gr$ be a set.
Then for all positive integers $n,m$ the following holds
\begin{equation}\label{f:Plunnecke}
    |nA-mA| \le K^{n+m} |A| \,.
\end{equation}
\label{l:Plunnecke}
\end{lemma}

We need in several quantitative versions
of the Balog--Szemer\'{e}di--Gowers Theorem.
The first symmetric variant
is due to T. Schoen \cite{schoen_BSzG}.

\begin{theorem}
    Let $A\subseteq \Gr$ be a set, $K\geq{1}$ and $\E(A)\geq{\frac{|A|^{3}}{K}}$.
    Then there is $A'\subseteq A$ such that
    $$
        |A'| \gg{\frac{|A|}{K}}\,,
    $$
    and
    $$
        |A'-A'| \ll K^4 |A'| \,.
    $$
\label{BSG}
\end{theorem}

\bigskip

Also we need in
a version
of Balog--Szemer\'{e}di--Gowers theorem in the asymmetric form, see \cite{TV}, Theorem 2.35.



\begin{theorem}
    Let $A,B\subseteq \Gr$ be two sets, $|B| \le |A|$, and $M\ge 1$ be a real number.
    Let also $L=|A|/|B|$ and $\eps \in (0,1]$ be a real parameter.
    Suppose that
\begin{equation}\label{cond:BSzG_as}
    \E (A,B) \ge \frac{|A| |B|^2}{M} \,.
\end{equation}
    Then there are two sets $H\subseteq \Gr$, $\L \subseteq \Gr$ and $z\in \Gr$ such that
\begin{equation}\label{f:BSzG_as_1}
    |(H+z) \cap B| \gg_\eps M^{-O_\eps (1)} L^{-\eps} |B| \,,
    \quad \quad
    |\L| \ll_{\eps} M^{O_\eps (1)} L^\eps \frac{|A|}{|H|} \,,
\end{equation}
\begin{equation}\label{f:BSzG_as_2}
    |H - H| \ll_{\eps} M^{O_\eps (1)} L^\eps \cdot |H| \,,
\end{equation}
     and
\begin{equation}\label{f:BSzG_as_3}
    |A\cap (H+\L)| \gg_{\eps} M^{-O_\eps (1)} L^{-\eps} |A| \,.
\end{equation}
\label{t:BSzG_as}
\end{theorem}

The next lemma is a special case of  Lemma 2.8 from \cite{SV}.
In particular, it gives us a connection between $\E_3 (A)$
and $\E (A,A_s)$, see e.g  \cite{SS1}.

\begin{lemma}
    Let $A\subseteq \Gr$ be a set.
    Then for every $k,l\in \N$
$$
    \sum_{s,t:\atop \|s\|=k-1,\,\, \|t\|=l-1} \E(A_s,A_t)=\E_{k+l}(A)
    \,.
$$
    In particular,
$$
    \E_3 (A) = \sum_s \E (A,A_s) \,.
$$
\label{l:E_3_A_s}
\end{lemma}

Now recall a lemma from \cite{SS3}, \cite{s_mixed}.

\begin{lemma}
\label{corpop}
    Let $A$ be a subset of an abelian group, $P_* \subseteq A-A$.
    Then
    \begin{equation*}
        \sum_{s\in P_*} |A\pm A_s| \geq \frac{\sigma^2_{P_*} (A) |A|^2}{\E_3(A) }
    \end{equation*}
    and
\begin{equation}\label{f:corpop2}
    \E_3 (P_*,A,A) \cdot \E_3 (A) \ge \frac{\E^2 (A) \sigma^4_{P_*} (A)}{|A|^{6}}
        \,.
\end{equation}
\end{lemma}

\bigskip

Let also give a simple
Corollary 18 from \cite{s_ineq}.

\begin{lemma}
    Let $A\subseteq \Gr$ be a set.
    Then
$$
    \sum_{s} \frac{|A_s|^2}{|A\pm A_s|} \le \frac{\E_3 (A)}{|A|^2}
    \,.
$$
\label{l:E_3_weight}
\end{lemma}

We give a
small
generalization of  Proposition 11 from \cite{SS1},
see also \cite{M_R-N_S}.

\begin{lemma}
    Let $A\subseteq \Gr$ be a set, $n,m\ge 1$ be positive integers.
    Then
    \begin{equation}\label{tmp:12.05.2014_1}
        |A^{n+m} - \D(A)| \ge |A|^m |A^{n} - \D(A)| \,,
    \end{equation}
    and
    \begin{equation}\label{tmp:12.05.2014_2}
        |A^{n+m} + \D(A)| \ge |A|^m \max\{ |A^{n} + \D(A)|, |A^{n} - \D(A)| \} \,.
    \end{equation}
    In particular,
    \begin{equation}\label{f:A^2_pm_p1}
        |A^2 - \D(A)| = \sum_{s\in A-A} |A-A_s| \ge |A| |A-A| \,,
    \end{equation}
    and
    \begin{equation}\label{f:A^2_pm_p2}
        |A^2 + \D(A)| = \sum_{s\in A-A} |A+A_s| \ge |A| \max\{ |A+A|,|A-A| \} \,.
    \end{equation}
\label{l:A^2_pm}
\end{lemma}
\begin{proof}
    In view of \cite{SS1}, Proposition 11 it remains to prove
    the second bound from (\ref{tmp:12.05.2014_2})
    in the case $m=1$ only, namely, that
    $|A^{n+1} + \D(A)| \ge |A| |A^{n} - \D(A)|$, $n\ge 1$.
    But $(a_1+a,\dots,a_{n}+a,a_{n+1}+a) \in A^{n+1} + \D(A)$ iff $a_{n+1}\in A_{s_1,\dots,s_n}$,
    where $s_j = a_j-a_{n+1}$, $(s_1,\dots,s_n) \in A^n - \D(A)$.
    Thus
$$
    |A^{n+1} + \D(A)| = \sum_{(s_1,\dots,s_n) \in A^n - \D(A)} |A+A_{s_1,\dots,s_n}|
        \ge
            |A| |A^{n} - \D(A)|
$$
    and the result follows.
$\hfill\Box$
\end{proof}

\bigskip

We will use very often the Katz--Koester trick \cite{kk}
\begin{equation}\label{f:KK_trick}
    A - A_s \subseteq (A-A)_{-s}\,, \quad \quad \quad A + A_s \subseteq (A+A)_s \,,
\end{equation}
and its generalization (see e.g. \cite{SV})
\begin{equation}\label{f:KK_trick_new}
    A - A_{\v{x}} \subseteq (A-A)_{-\v{x}} \,, \quad \quad \quad A + A_{\v{x}} \subseteq (A+A)_{\v{x}} \,.
\end{equation}

\bigskip

Finally, recall some results from \cite{s_mixed}.
We begin with an analog of a definition from \cite{s_doubling}.

\begin{definition}
    Let $\a > 1$ be a real number, $\beta,\gamma \in [0,1]$.
    A set $A\subseteq \Gr$ is called $(\a,\beta,\gamma)$--connected
    if for any $B \subseteq A$, $|B| \ge \beta|A|$
    the following holds
    $$
        \E_\a (B) \ge \gamma \left( \frac{|B|}{|A|} \right)^{2\a} \E_\a (A) \,.
    $$
\label{def:conn}
\end{definition}

Thus, a set from Theorem \ref{t:BK_structural} is a $(2,\beta,\gamma)$--connected set
with $\beta,\gamma \gg 1$.
The H\"{o}lder inequality implies that if
$\E_\a (A) \le \gamma^{-1} |A|^{2\a} |A-A|^{1-\a}$ then $A$ is $(\a,\beta,\gamma)$--connected for any $\beta$.
As was proved in  \cite{s_doubling} that for $\a=2$ {\it every} set $A$ always contains large connected subset.
For
integers $\a>2$, see the Appendix.

\bigskip

Our first lemma from \cite{s_mixed} (where some operators were used in the proof)
is about a nontrivial lower bound for $\E_s (A)$, $s\in [1,2]$
in terms of $\E(A)$.

\begin{lemma}
    Let $A\subseteq \Gr$ be a set, and $\beta,\gamma \in [0,1]$.
    Suppose that $A$ is $(2,\beta,\gamma)$--connected with $\beta \le 1/2$.
    Then for any $s\in [1,2]$ the following holds
\begin{equation}\label{f:connected}
    \E_s (A) \ge 2^{-5} \gamma |A|^{1-s/2} \E^{s/2} (A) \,.
\end{equation}
\label{l:connected}
\end{lemma}

The second lemma from \cite{s_mixed} is about an upper bound for eigenvalues of some operators.
To avoid of using the operators notation we formulate the result in the following way.

\begin{lemma}
    Let $A\subseteq \Gr$ be a set.
    Then for an arbitrary function $f : A \to \C$ one has
\begin{equation}\label{f:eigen_A}
        \E(A,f) \le \E^{1/2}_3 (A) \| f\|_2^2 \,.
\end{equation}
    Further, there is a set $A'\subseteq A$, $|A'| \ge |A|/2$, namely,
\begin{equation}\label{f:A'_def}
    A' := \{ x ~:~ ((A*A) \c A) (x) \le 2 \E(A) |A|^{-1} \}
\end{equation}
    such that for any
    function $f : A' \to \C$ the following holds
\begin{equation}\label{f:eigen_A'}
    \E(A,f) \le \frac{2\E(A)}{|A|} \cdot \| f\|_2^2 \,.
\end{equation}
    Moreover for any even real function $g$ there is a set $A'\subseteq A$, $|A'| \ge |A|/2$
    such that for any
    function $f : A' \to \C$ the following holds
\begin{equation}\label{f:eigen_A''}
    \sum_x g(x) (\ov{f} \c f)(x) = \sum_x g(x) (f\c \ov{f})(x)
        \le 2 |A|^{-1} \sum_x g(x) (A\c A) (x) \cdot \| f\|_2^2 \,.
\end{equation}
\label{l:eigen_A'}
\end{lemma}

Note that
for the characteristic functions $f$ of
sets from
$A$
bound
(\ref{f:eigen_A}) can be obtained using the Cauchy--Schwarz inequality.
Further, estimate  (\ref{f:eigen_A''}) is a generalization of (\ref{f:eigen_A'})
which was
proved
in \cite{s_mixed}, see Lemma 44. Bound (\ref{f:eigen_A''}) can be obtained in a similar way.

\bigskip


We finish the section noting a generalization of formula (\ref{f:corpop2}) of Lemma \ref{corpop}.
That is just a part of  Lemma 4.2 from \cite{s_mixed}.

\begin{lemma}
    Let $A,B\subseteq \Gr$ be finite sets, $S\subseteq \Gr$ be a set such that
    $A+B \subseteq S$.
    Suppose that $\psi$ is a
    function on $\Gr$.
    Then
\begin{equation}\label{f:T_A,B}
    |B|^2 \cdot \left( \sum_{x} \psi(x) (A\c A)(x) \right)^2
        \le
            \E_3(B,A) \sum_{x} \psi^2 (x) (S\c S)(x) \,.
\end{equation}
\label{l:T_A,B}
\end{lemma}

\section{Structural results}
\label{sec:structural}


In this section we obtain several general structural results,
some of which have applications to sum--products phenomenon, for example.
These results are closely related to the Balog-Szemer\'{e}di-Gowers Theorem, and we adopt the convention of writing $\Gr$ as an additive group.
The proofs follow the arguments from \cite{SS1} and \cite{SS2}.

Now we formulate the first result of the section.

\begin{proposition}
    Let $A\subseteq \Gr$ be a finite set, and $M\ge 1$, $\eta \in (0,1]$ be real numbers.
    Let $\E(A) = |A|^3/K$.
    Suppose that for some set $P\subseteq A-A$ the following holds
\begin{equation}\label{cond:eta}
    \sum_{s\in P} (A\c A) (s) = \eta |A|^2 \,,
\end{equation}
    and
\begin{equation}\label{cond:M}
    \sum_{s\in P} |A \pm A_s| \le M K |A|^2 \,.
\end{equation}
    Then
    for any $\eps \in (0,1)$,
    there are two sets $H\subseteq \Gr$, $\L \subseteq \Gr$ and $z\in \Gr$ such that
\begin{equation}\label{f:E_3_and_E_critical_1'}
    |(H+z) \cap A| \gg_{M,\eta^{-1},K^\eps} \frac{\E(A)}{|A|^2} \,,
    \quad \quad
    |\L| \ll_{M,\eta^{-1},K^\eps} \frac{|A|}{|H|} \,,
\end{equation}
\begin{equation}\label{f:E_3_and_E_critical_2'}
    |H - H| \ll_{M,\eta^{-1},K^\eps} |H| \,,
\end{equation}
     and
\begin{equation}\label{f:E_3_and_E_critical_3'}
    |A\bigcap (H+\L)| \gg_{M,\eta^{-1},K^\eps} |A| \,.
\end{equation}
\label{p:A^2-D(A)}
\end{proposition}
\begin{proof}
Using Lemma \ref{corpop} with $P_*=P$, we see that
\begin{equation}\label{tmp:08.05.2014_1}
    \E_3 (A) \ge \frac{\eta^2 |A|^4}{M K} \,.
\end{equation}
Note that
\begin{align*}
\sum_{s ~:~ |A_s|<\frac{|A|\eta^2}{2KM}} \E(A,A_s)&\leq{\left(\frac{|A|\eta^2}{2KM}\right)\sum_s|A_s||A|}
\\&=\frac{|A|^4\eta^2}{2KM} \,.
\end{align*}
Applying Lemma \ref{l:E_3_A_s}, that is the formula $\E_3 (A) = \sum_s \E(A,A_s)$,
combining with (\ref{tmp:08.05.2014_1}), we get
\begin{equation}\label{TMP:05.10.2013}
    \sum_{s ~:~ |A_s| \ge 2^{-1} \eta^2 M^{-1} K^{-1} |A| } \E(A,A_s) \ge \frac{\eta^2 |A|^4}{2M K} \,.
\end{equation}
Put
$$
    \mu := \max_{s ~:~ |A_s| \ge 2^{-1} \eta^2 M^{-1} K^{-1} |A| } \frac{\E(A,A_s)}{|A| |A_s|^2 } \,.
$$
Using (\ref{TMP:05.10.2013}), we have
$$
    \mu |A| \E (A) \ge \mu |A| \cdot \sum_{s ~:~ |A_s| \ge 2^{-1} \eta^2 M^{-1} K^{-1} |A| } |A_s|^2
        \ge
            \frac{\eta^2 |A|^4}{2M K} \,.
$$
Thus, $\mu \ge \frac{\eta^2}{2M}$.
Hence there is an $s$
with
$|A_s| \ge 2^{-1} \eta^2 M^{-1} K^{-1} |A|$
and
such that\\
${\E(A,A_s) \ge 2^{-1} M^{-1} \eta^2|A||A_s|^2}$.
Applying the asymmetric version of Balog--Szemer\'{e}di--Gowers Theorem \ref{t:BSzG_as},
we find two sets $\L,H$ such that
(\ref{f:E_3_and_E_critical_1'})---(\ref{f:E_3_and_E_critical_3'}) take place.
This completes the proof.
$\hfill\Box$
\end{proof}

\bigskip
We write the fact that sets $A,H,\L$ satisfy (\ref{f:E_3_and_E_critical_1'})---(\ref{f:E_3_and_E_critical_3'})
with $\eta \gg 1$ as
\begin{equation}
    A \approx_{M,K^\eps} \L \dotplus H.
\label{approxdefn}
\end{equation}
Note that the degree of polynomial dependence in formula (\ref{approxdefn}) is a function on $\eps$.

\begin{example}
    Let $H\subseteq \mathbf{F}_2^n$ be a subspace and $\Lambda \subseteq \mathbf{F}_2^n$ be a dissociated set (basis).
    Put $A=H \dotplus \Lambda$, where $\dotplus$ means the direct sum, and $|\Lambda| = K$.
    Detailed discussion of the example can be found, e.g. in \cite{s_mixed}.
    If $s\in H$ then $A_s = A$ and hence $A+A_s = A+A$.
    If $s\in (A+A) \setminus H$ then $A_s$ is the disjoint union of two shifts of $H$
    and thus $|A+A_s| \le 2|A|$.
    Whence
    $$
        \sum_{s\in A+A} |A+A_s| \le |H| |A+A| + 2|A+A||A| \ll K|A|^2 \,,
    $$
    and $\E(A) \sim |A|^3 /K$.
    It means that condition (\ref{cond:M}) takes place in the case
    $A=H \dotplus \Lambda$.
\end{example}

Taking $P=A-A$ and applying Proposition \ref{p:A^2-D(A)}
as well as formulas (\ref{f:A^2_pm_p1}), (\ref{f:A^2_pm_p2}) of Lemma \ref{l:A^2_pm},
we obtain the following consequence.

\begin{corollary}
    Let $A\subseteq \Gr$ be a set, $M\in{\mathbb{R}}$, $\eps \in (0,1)$ and $\E(A) = |A|^3/K$.
    Then either
    $$
        |A^2 \pm \D(A)| \ge M K |A|^2
    $$
    or
    $A \approx_{M,K^{\eps}} \L \dotplus H$.
\label{c:A^2-D(A)}
\end{corollary}

Note that for any set $A\subseteq \Gr$ with $\E(A) = |A|^3/K$
the inequality $|A^2 \pm \D(A)| \ge K |A|^2$ follows from Lemma \ref{corpop}
and a trivial
estimate
$\E_3(A) \le |A| \E(A)$.
We will deal
with the reverse condition $\E_3(A) \gg |A| \E(A)$
in Proposition \ref{p:E_3_and_E_critical}
and Theorem \ref{t:E_3_and_E_critical} below.

The next corollary shows that if  a set $A$ is not close
to a set of the form $\L \dotplus H$ then there is some imbalance
(in view of Pl\"{u}nnecke--Ruzsa inequality (\ref{f:Plunnecke}))
between the doubling constant
and the additive energy of $A$ or $A-A$.

\begin{corollary}
    Let $A\subseteq \Gr$ be a set, $M$, $\eps \in (0,1)$ be real numbers and
    $|(A-A) \pm (A-A)| \gg |A-A|^3 / |A|^2$.
    Then either
    $$
        \E(A) \gg \frac{M^{1/2} |A|^4}{|A-A|}
            \quad \mbox{ or } \quad \E(A-A) \gg \frac{M |A-A|^4}{|(A-A) \pm (A-A)|}
    $$
    or
    $A \approx_{M,K^{\eps}} \L \dotplus H$, where
    $K = |A-A|/|A|$.
\label{c:3A}
\end{corollary}
\begin{proof}
Put $D=A-A$.
Suppose that $\E(A) \ll \frac{M^{1/2} |A|^4}{|D|}$ because otherwise we are done.
In view of Corollary \ref{c:A^2-D(A)} one can assume that $\sum_{s} |A-A_s| \ge M^{1/2} |A| |D|$.
Thus, by the Katz--Koester trick (\ref{f:KK_trick})
\begin{equation}\label{f:KK_trick'}
    A - A_s \subseteq (A-A)_{-s}\,, \quad \quad \quad A + A_s \subseteq (A+A)_s \,,
\end{equation}
and the Cauchy--Schwarz inequality, we get
$$
    |D| \E(D) \ge \left( \sum_{s\in D} (D\c D) (s) \right)^2 \ge \left( \sum_{s\in D} |A-A_s| \right)^2
        \ge (M^{1/2} |A| |D|)^2 \,.
$$
Hence
$$
    \E(D) \ge M^{} |D|^{} |A|^2 \gg \frac{M^{} |D|^4}{|(A-A) \pm (A-A)|}
$$
where the assumption of the corollary has been used.
This completes the proof.
$\hfill\Box$
\end{proof}

\bigskip

The quantities $\E(A\pm A)$ (and hence $|A^2 \pm \D(A)|$ in view of Lemma \ref{l:A^2_pm},
see also Proposition \ref{pr:E_k(D)_simple} below)
appear in sum--products results (in multiplicative form).
For example, in \cite{M_R-N_S} the following
theorem was proved.

\begin{theorem}
    Let $A,B\subseteq \R$ be finite sets.
Then
\begin{equation*}\label{f:main_intr_2_new}
    |B+AA|^3 \gg \frac{|B| \E^{\times} (AA)}{\log |A|} \,.
\end{equation*}
\end{theorem}

Here $\E^\times (A) := |\{ a_1 a_2 = a_3 a_4 ~:~ a_1, a_2, a_3, a_4 \in A \}|$.
Thus, by the obtained results, we have, roughly, that  either $\E^\times (A)$, $\E^\times (AA)$
can be estimated better then
by Lemma \ref{l:A^2_pm}, formulas (\ref{f:A^2_pm_p1}), (\ref{f:A^2_pm_p2})
or $A$ has the rigid structure $A \approx \L \cdot H$.
Usually the last case is easy to deal with.
Similar methods were used in \cite{M_R-N_S}.

\bigskip

Now we obtain another structural result.
Using Lemma \ref{corpop} as well as a trivial estimate $\E_3 (A) \le |A| \E(A)$
one can derive Proposition \ref{p:A^2-D(A)} from Proposition \ref{p:E_3_and_E_critical}
below.

\begin{proposition}
    Let $A\subseteq \Gr$ be a set, and $M\ge 1$ be a real number.
    Suppose that
\begin{equation}\label{cond:E_3_and_E_critical}
    \E_3 (A) \ge \frac{|A| \E(A)}{M} \,.
\end{equation}
    Then there is $A' \subseteq A$ such that
\begin{equation}\label{f:E_3_and_E_critical_1}
    |A'| \gg \frac{\E(A)}{|A|^2 (M \log M)^5}
\end{equation}
    and
\begin{equation}\label{f:E_3_and_E_critical_2}
    |A' - A'| \ll M^{15} \log^{16} M \cdot |A'| \,.
\end{equation}
    Further, take any $\eps \in (0,1)$ and
    put $K:=\frac{|A|^3}{\E (A)} $.
    Then $A \approx_{M,K^\eps} \L \dotplus H$.
\label{p:E_3_and_E_critical}
\end{proposition}
\begin{proof}
First of all prove (\ref{f:E_3_and_E_critical_1}),
(\ref{f:E_3_and_E_critical_2}). Let
$$
    P_j = \{ x ~:~ 2^{j-1} |A| / (2^{2} M)
        < |A_x| \le 2^{j} |A| / (2^{2} M) \} \,, \quad j\in [L] \,,
$$
where $L=[\log(4M)]$. By the pigeonhole principle there is $j\in [L]$
such that
$$
    \frac{|A| \E(A)}{2M L} \le \frac{\E_3 (A)}{2L} \le \sum_{x\in P_j} |A_x|^3 \,.
$$
Put $P=P_j$ and $\D =  2^{j} |A| / (2^{2} M)$.
Thus
\begin{equation}\label{tmp:04.02.2013_2}
    \frac{8M \E(A)}{2^{2j} |A|L} \le \sum_x P(x) (A\c A) (x) = \sum_x A(x) (A\c P) (x) \,.
\end{equation}
Hence, by the Cauchy--Schwarz inequality
\begin{equation}\label{tmp:05.02.2013_1}
    \frac{2^{6} M^2 \E^2 (A)}{2^{4j} |A|^3 L^2} \le \E(A,P) \le (\E (A))^{1/2} (\E (P))^{1/2} \,.
\end{equation}
Note that
$$\E(A)\geq{\sum_{x\in{P}}|A_x|^2}\geq{\frac{|P||A|^22^{2j-2}}{2^4M^2}},$$
and therefore
\begin{equation}
|P| \le 2^{6} 2^{-2j} M^{2} \E(A) |A|^{-2}.
\label{Pbound}
\end{equation}
 It follows that
\begin{equation}\label{tmp:08.02.2013_4}
    \E(P) \ge \frac{2^{12} M^4 (\E (A))^3}{2^{8j} |A|^6 L^4} \ge \frac{|P|^3}{2^{6} M^2 2^{2j} L^4}
        \ge
            \frac{|P|^3}{2^{10} M^4 L^4} := \mu |P|^3 \,.
\end{equation}
By Theorem \ref{BSG} there is $P'\subseteq P$ such that
$|P'| \gg \mu |P|$ and $|P'-P'| \ll \mu^{-4} |P'|$.
Note that
$$\sum_{x\in{A}} (A\c P') (x)=\sum_{x\in{P'}} (A\c A) (x)\geq{\frac{|P'|2^{j-3}|A|}{M}},$$
and so there exists  $x\in A$ such that the set $A' := A\cap (P'+x)$ has the size at
least $|P'| 2^{j-3} M^{-1}$. We have
\begin{equation}\label{tmp:08.02.2013_5}
    |A'-A'| \le |P' - P'| \ll \mu^{-4} |P'| \ll \mu^{-4} 2^{-j} M |A'| \ll M^{15} L^{16} |A'| \,.
\end{equation}
Finally, from (\ref{tmp:04.02.2013_2}), say, one has
\begin{equation}\label{tmp:06.02.2013_1}
    |P| \gg \frac{M^2 \E(A)}{2^{3j} |A|^2 L} \gg \frac{\E(A)}{M |A|^2 L}
\end{equation}
and because of
$$
    |A'| \ge |P'| 2^{j-3} M^{-1} \gg \mu |P| \cdot 2^{j-3} M^{-1}
$$
the result follows.

To obtain
(\ref{f:E_3_and_E_critical_1'})---(\ref{f:E_3_and_E_critical_3'}),
that is $A \approx_{M,K^\eps} \L \dotplus H$, $K = |A|^3 \E^{-1} (A)$,
note that by the first inequality of (\ref{tmp:05.02.2013_1}) and the bound
$|P| \le 2^{6} 2^{-2j} M^{2} \E(A) |A|^{-2}$, we have
$$
    \E(A,P) \ge \frac{2^{6} M^2 \E^2 (A)}{2^{4j} |A|^3 L^2}
        \ge
            \frac{|A| |P|^2}{2^6 L^2 M^2} \,.
$$
Also, by the definition of the number $K$, and inequality  \eqref{tmp:06.02.2013_1} the following holds
$$|A|/|P| \ll ML K \ll_M K.$$
Applying the asymmetric version of
Balog--Szemer\'{e}di--Gowers Theorem \ref{t:BSzG_as} with $A=A$, $B=P$, and
recalling (\ref{tmp:06.02.2013_1}), we obtain the required inequalities,
excepting the first inequality of (\ref{f:E_3_and_E_critical_1'}),
where it remains to
 replace $P$ by $A$.
Put $H' = (H+z) \cap P$.
We have $|H'| \gg_{M,K^\eps} |P|$.
Thus, by the definition of the number $\D$ and estimate (\ref{tmp:06.02.2013_1}), we obtain
\begin{equation}\label{tmp:04.05.2014_1}
    \sum_{x\in A} (A \c H') (x) = \sum_{x\in H'} (A\c A) (x) \ge 2^{-1} \D |H'|
        \gg_{M,K^\eps} \D |P|
        \gg_{M,K^\eps} \frac{\E (A)}{|A|} \,.
\end{equation}
Hence there is $w\in A$ such that
\begin{equation}\label{tmp:04.05.2014_2}
    |(H+w) \cap A| \ge |(H'+w) \cap A| \gg_{M,K^\eps} \frac{\E (A)}{|A|^2} \,.
\end{equation}
This completes the proof.
$\hfill\Box$
\end{proof}

\bigskip

Assumption (\ref{cond:E_3_and_E_critical}) of the Proposition \ref{p:E_3_and_E_critical}
is a generalisation of the usual condition $\E (A) \ge \frac{|A|^3}{M}$
(because of $\E(A) |A|^3 M^{-1} \le \E^2 (A) \le \E_3 (A) |A|^2$)
and $\E_3 (A) \ge \frac{|A|^4}{M}$
(because of $\E_3 (A) \ge |A|^4 M^{-1} \ge |A| \E(A) M^{-1}$).
Further, one can check that the same consequences
(\ref{f:E_3_and_E_critical_1'})---(\ref{f:E_3_and_E_critical_3'}) hold
if we replace condition (\ref{cond:E_3_and_E_critical}) by
$\E_s (A) \ge |A| \E_{s-1} (A) / M$, $s\ge 3$.
Let us write the correspondent result.

\begin{theorem}
    Let $A\subseteq \Gr$ be a set, $s\ge 3$ be a positive integer, and $M\ge 1$ be a real number.
    Suppose that
\begin{equation}\label{cond:E_3_and_E_criticalt}
    \E_s (A) \ge \frac{|A| \E_{s-1} (A)}{M} \,.
\end{equation}
    Take any $\eps \in (0,1)$ and put $K:=\frac{|A|^s}{\E_{s-1} (A)} $.
    Then $A \approx_{s,\, M,\, K^\eps} \L \dotplus H$, $|H| \gg_{s,\, M,\, K^\eps} |A|/K$.
\label{t:E_3_and_E_critical}
\end{theorem}
\begin{proof}
The arguments almost repeat the proof of Proposition \ref{p:E_3_and_E_critical}, so we
skip some details.
Using dyadic pigeonholing and the assumption, we find $P\subseteq A-A$, $\D < |A_x|\le 2\D$, $x\in P$
with
$$
    M^{-1} |A| \E_{s-1} (A) \le \E_{s} (A) \ll_{\log M} \sum_{x\in P} |A_x|^s \ll_{\log M} \D^{s-1} \sigma_P (A) \,.
$$
Thus, by the Cauchy--Schwarz inequality
$$
    M^{-2} |A| \E^2_{s-1} (A) \D^{-2(s-1)} \ll_{\log M} \E(A,P) \,.
$$
On the other hand $|P| \D^{s-1} \le \E_{s-1} (A)$ and hence
$$
    \E(A,P) \gg_{\log M} M^{-2} |A| |P|^2 \,.
$$
Note that by (\ref{cond:E_3_and_E_criticalt}) and our choice of the set $P$, we have
$$
    \frac{|A|}{|P|} \ll_{M} \frac{\D^{s}}{\E_{s-1} (A)} \le \frac{|A|^s}{\E_{s-1} (A)} \,,
$$
and, again,
$$
    |P| \sim_{\log M} \E_s (A) \D^{-s} \ge  M^{-1} |A| \E_{s-1} (A) \D^{-s} = \frac{|A|^{s+1}}{MK \D^{s}}
        \ge
            \frac{|A|}{MK} \,.
$$
After that apply the asymmetric version of Balog--Szemer\'{e}di--Gowers Theorem \ref{t:BSzG_as}
and an analog of the arguments from (\ref{tmp:04.05.2014_1})---(\ref{tmp:04.05.2014_2}).
This concludes
the proof.
\end{proof}
$\hfill\Box$

\bigskip

The more general assumption
$\E_s (A) \ge |A|^k \E_{s-k} (A)  / M^k$
implies that for some $j \in [k]$ one has $\E_{s-j+1} (A) \ge |A| \E_{s-j}(A)  / M$.
Thus, we have considered
the common
case.
Note, finally, that estimates (\ref{f:E_3_and_E_critical_1}), (\ref{f:E_3_and_E_critical_2})
are the best possible.
Indeed, take $\Gr = \mathbf{F}^n_2$, $A=H\dotplus \L$, where $H\le \mathbf{F}^n_2$ is a linear subspace
and $\L$ is a dissociated set (basis).

\bigskip

Now we can prove a criterium for sets having critical relation between $\E (A) $ and $\E_3 (A)$ energies.

\begin{theorem}
    Let $A\subseteq \Gr$ be a set, and $M\ge 1, \eps \in (0,1)$ be real numbers.
    Put $K:=\frac{|A|^3}{\E (A)}$.
    Then
\begin{equation*}\label{cond:E_3_and_E_critical_cor}
    \E_3 (A) \gg_{M,K^\eps} |A| \E(A)
\end{equation*}
    iff
$$
    A \approx_{M,K^\eps} \L \dotplus H \,.
$$
\label{t:H+L_description}
\end{theorem}
\begin{proof}
    In view of Proposition \ref{p:E_3_and_E_critical} it remains to prove that if
$A \approx_{M,K^\eps} \L \dotplus H$ then $\E_3 (A) \gg_{M,K^\eps} |A| \E(A)$.
    Put $A_1 = A\cap (H+\L)$.
    We have $|A_1| \gg_{M,K^\eps} |A|$.
    Then
$$
    |A|^2 |H|^2 \ll_{M,K^\eps} |A_1|^2 |H|^2 \le \E(A_1,H) |A_1-H| \le \E(A,H) |\L| |H-H| \ll_{M,K^\eps} \E(A,H) |A| \,.
$$
    Thus
$$
    (|A| |H|^2)^3 \ll_{M,K^\eps} \left( \sum_x (A\c A)(x) (H\c H) (x) \right)^3
        \le
            \E_3 (A) \E^2_{3/2} (H)
                \le
                    \E_3 (A) |H|^5 \,.
$$
In other words
$$
    \E_3 (A) \gg_{M,K^\eps} |A|^3 |H| \gg_{M,K^\eps} \E(A) |A| \,.
$$
To get the last estimate
we have used
the fact $|H| \gg_{M,K^\eps} \E (A) |A|^{-2}$.
 This completes the proof.
$\hfill\Box$
\end{proof}

\bigskip

Recall that
$$
   \T_k (A) :=| \{ a_1 + \dots + a_k = a'_1 + \dots + a'_k  ~:~ a_1, \dots, a_k, a'_1,\dots,a'_k \in A \} |
   \,.
$$
We conclude the section proving
a ``dual"\, analogue of Proposition \ref{p:E_3_and_E_critical},
that is we replace the condition on $\E_3 (A)$ with a similar condition for $\T_4 (A)$
and moreover for $\T_s (A)$.
Again, the proof follows the arguments from \cite{SS1}.

\begin{theorem}
    Let $A\subseteq \Gr$ be a set, and $M\ge 1$ be a real number.
    Suppose that
\begin{equation}\label{cond:T_3_and_E_critical}
    \T_4 (A) \ge \frac{|A|^4 \E(A)}{M} \,.
\end{equation}
    Then there is $A' \subseteq A$ such that
\begin{equation}\label{f:T_3_and_E_critical_1}
    |A'| \gg \frac{|A|}{M^3 \log^{\frac{16}{3}} |A|} \,,
\end{equation}
    and
\begin{equation}\label{f:T_3_and_E_critical_2}
    |nA' - mA'| \ll (M^3 \log^4 |A|)^{4(n+m)} M |A'| \cdot \frac{|A|^3}{\E(A)}
\end{equation}
for every $n,m\in \N$.

Moreover, if
\begin{equation}\label{cond:T_3_and_E_critical_s}
    \T_{2s} (A) \ge \frac{|A|^{2s} \T_s (A)}{M} \,,
\end{equation}
$s\ge 2$ then formulas (\ref{f:T_3_and_E_critical_1}),
(\ref{f:T_3_and_E_critical_2}) take place.
Conversely, bounds  (\ref{f:T_3_and_E_critical_1}), (\ref{f:T_3_and_E_critical_2}) imply that
$\T_{2s} (A) \gg_{M,\,\log |A|,\,s} |A|^{2s} \T_s (A)$.
\label{t:T_3_and_E_critical}
\end{theorem}
\begin{proof}
Put
$\T_s = \T_s (A)$, $\T_{2s} = \T_{2s} (A)$, $a=|A|$,
$L_s = [\log (16Ma^{2s-1}/\T_s)] \ll_s \log a$.
Let
$$
    P_j = \{ x ~:~ 2^{j-1} \T_s / (2^{4} M a^s)
        < (A *_{s-1} A) (x) \le 2^{j} \T_s / (2^{4} M a^s) \} \,, \quad j\in [L_s] \,.
$$
Put $f_j (x) = P_j (x) (A *_{s-1} A) (x)$.
Thus,
\begin{equation}\label{tmp:04.05.2014_3}
    (A *_{s-1} A) (x) = \sum_{j=1}^{L_s} f_j (x) + \Omega (x) (A *_{s-1} A)(x) \,,
\end{equation}
where
$\Omega = \{ x ~:~ (A *_{s-1} A) (x) \le 2^{-4} M^{-1} \T_s a^{-s} \}$.
Substituting formula (\ref{tmp:04.05.2014_3}) into the identity
$$
    \T_{2s} (A) = \sum_{x} ((A *_{s-1} A) \c (A *_{s-1} A))^2 (x)
$$
and using assumptions (\ref{cond:T_3_and_E_critical}), (\ref{cond:T_3_and_E_critical_s}),
combining with the definition of sets $P_j$, $\Omega$,
we have
$$
    2^{-1} \T_{2s} (A) \le \sum_{j_1,j_2,j_3,j_4 = 1}^{L_s} \sum_x ( f_{j_1} \c f_{j_2} ) (x) ( f_{j_3} \c f_{j_4} ) (x) \,.
$$
Applying the H\"{o}lder inequality, we get
$$
    2^{-1} L_s^{-3} \T_{2s} (A) \le \sum_{j=1}^{L_s} \sum_x ( f_{j} \c f_{j} ) (x) ( f_{j} \c f_{j} ) (x) \,.
$$
By the pigeonhole principle there is $j\in [L_s]$ such that
\begin{equation}\label{tmp:08.02.2013_6}
    \frac{a^{2s} \T_s}{2M L_s^4} \le \frac{\T_{2s}}{2L_s^{4}}
        \le \sum_x ( f_{j} \c f_{j} ) (x) ( f_{j} \c f_{j} ) (x) \,.
\end{equation}
Put $P=P_j$, $f=f_j$ and $\D =  2^{j} \T_s / (2^{4} M a^s)$.
Thus
\begin{equation}\label{tmp:08.02.2013_3}
    \frac{a^{2s} \T_s}{2M L_s^4 \D^4} \le \E (P) \,.
\end{equation}
Clearly, $|P| \le 4\T_s \D^{-2}$.
Using the last inequality, the definition of the number $\D$
and bound (\ref{tmp:08.02.2013_3}), we obtain
$$
    \E (P) \ge \frac{a^{2s} \T_s}{2M L_s^4 \D^4} \ge |P|^3 \frac{a^{2s} \D^2 }{2^7 M L_s^4 \T_s^2}
        \ge
            |P|^3 \frac{2^{2j}}{2^{15} M^3 L_s^4}
                \ge
                    \frac{|P|^3}{2^{15} M^3 L_s^4} := \mu |P|^3 \,.
$$
To estimate the size of $P$ we note by (\ref{tmp:08.02.2013_3}) that
\begin{equation}\label{tmp:08.02.2013_7}
    |P|^3 \ge \frac{a^{2s} \T_s}{2M L_s^4 \D^4} \,.
\end{equation}
After that use
arguments (\ref{tmp:08.02.2013_4})---(\ref{tmp:08.02.2013_5})
of the proof of Proposition \ref{p:E_3_and_E_critical}.
By Theorem \ref{BSG} there is $P'\subseteq P$ such that
$|P'| \gg \mu |P|$ and $|P'-P'| \ll \mu^{-4} |P'|$.
Applying Pl\"{u}nnecke--Ruzsa inequality (\ref{f:Plunnecke}), we obtain
\begin{equation}\label{tmp:08.02.2013_8}
    |nP'-mP'| \ll \mu^{-4(n+m)} |P'|
\end{equation}
for every $n,m\in \N$.
We have
$$
    \D |P'| \le \sum_x (A *_{s-1} A) (x) P'(x) = \sum_{x_1,\dots,x_{s-1}\in A} (A \c P') (x_1+\dots+x_{s-1}) \,.
$$
By (\ref{tmp:08.02.2013_7})
and the definition of the number $\D$
there is $x\in (s-1)A$ such that the set $A' := A\cap (P'-x)$ has the size at
least
$$
    |A'|
        \gg
            |P'| \D a^{-(s-1)}
                \ge
            \mu |P| \D a^{-(s-1)}
                \gg
                    \mu \left( \frac{\T_s}{M L_s^4 \D a^{s-3}} \right)^{1/3}
                        \gg
                            \frac{2^{5j/3} a}{M^3 L_s^{16/3}}
                                \gg
                                    \frac{a}{M^3 L_s^{16/3}} \,.
$$
We have by (\ref{tmp:08.02.2013_8}) that
\begin{equation*}\label{tmp:08.02.2013_5'}
    |nA'-mA'| \le |nP' - mP'| \ll \mu^{-4(n+m)} |P'|
        \ll \mu^{-4(n+m)} |A'| a^{s-1} \D^{-1} \ll \mu^{-4(n+m)} M |A'| \cdot \frac{a^{2s-1}}{\T_s}
\end{equation*}
for every $n,m\in \N$.

Conversely, applying bound (\ref{f:T_3_and_E_critical_2}) with $n=m=s$, combining with the Cauchy--Schwarz inequality,
we obtain
$$
    |A'|^{4s} \le \left( \sum_x ( (A' *_{s-1} A') \c (A' *_{s-1} A')  ) (x) \right)^2
        \le
            \T_{2s} (A') |sA'-sA'|
                \ll_s
$$
$$
                \ll_s
                    (M^{3} \log^{4} |A|)^{8s} M \cdot |A'| \frac{|A|^{2s-1}}{\T_s (A)} \T_{2s} (A) \,.
$$
Using (\ref{f:T_3_and_E_critical_1}), we get
$$
    \T_{2s} (A) \gg_s \T_s (A) |A'|^{4s-1} |A|^{-(2s-1)} (M^{3} \log^{4} |A|)^{-8s} M^{-1}
        \gg_s
            \T_s (A) |A|^{2s} M^{-36s+2} \log^{-54s} |A| \,.
$$
In other words, $\T_{2s} (A) \gg_{M,\,\log |A|,\,s} |A|^{2s} \T_s (A)$.
This completes the proof.
$\hfill\Box$
\end{proof}

\bigskip

So, we have proved in Theorem \ref{t:T_3_and_E_critical} that, roughly speaking, $A'-A'$
is a set with small (in terms of the parameter $M$) doubling and vice versa.
Note,
that we need in multiple $|A|^3 \E^{-1} (A)$ in (\ref{f:T_3_and_E_critical_2}),
because by (\ref{f:T_3_and_E_critical_1}) and the Cauchy--Schwarz inequality,
we have the same lower bound for $|A'-A'|$.
Thus, $A'$ does not equal to a set with small doubling but $A'-A'$ does.
Results of such a sort were obtained in \cite{SS1}, \cite{s_ineq} and \cite{s_mixed}.

It is easy to see that an analog of Theorem \ref{t:T_3_and_E_critical}
takes place if one replace (\ref{cond:T_3_and_E_critical})
onto condition
$\T_s (A) \ge |A|^{2(s-2)} \E(A) / M$,
where $s$ is an even number, $s\ge 4$,
and, further,  even  more general relations between $\T_k$ energies
can be reduced to the last case and Theorem \ref{t:T_3_and_E_critical}
 via a trivial estimate
$ \T_s (A) \le |A|^2 \T_{s-1} (A)$.
We do not need in such generalizations in the paper.

\section{Sumsets: preliminaries}
\label{sec:sumsets}

Let $A \subseteq \Gr$ be a set.
Before studying the energies of sumsets or difference sets we concentrate on a following related question,
which was asked to the author by Tomasz Schoen.
Namely, what can be proved nontrivial concerning lower bounds for $|A\pm A_s|$, $s\neq 0$?
The connection with $A\pm A$ is obvious in view of Katz--Koester trick (\ref{f:KK_trick}).
We start  with a result in the direction.

\begin{theorem}
    Let $A\subseteq \Gr$ be a set.
    If $\E_3 (A) \ge 2|A|^3$ then
\begin{equation}\label{f:E_3(D)_0}
        \left( \max_{s\neq 0} |A\pm A_s| \right)^3 \gg \frac{|A|^{10}}{|A-A| \E^2 (A)} \,.
\end{equation}
    Now suppose that $A$ is $(3,\beta,\gamma)$--connected with $\beta \le 1/2$,
    and $\E_3 (A) \ge 2^4 \gamma^{-1} |A|^3$.
    Then
\begin{equation}\label{f:E_3(D)_0'}
    \left( \max_{s\neq 0} |A\pm A_s| \right)^2 \gg \gamma \frac{|A|^{5}}{\E (A)} \,.
\end{equation}
\label{t:E_3(D)-}
\end{theorem}
\begin{proof}
    Write $\E = \E(A)$, $\E_3 = \E_3 (A)$, and $a=|A|$.
    Let us begin with (\ref{f:E_3(D)_0}).
    Denote by $\o$ the maximum in (\ref{f:E_3(D)_0}).
    By the Cauchy--Schwarz inequality  and formula (\ref{f:eigen_A}) of Lemma \ref{l:eigen_A'}, we obtain
\begin{equation}\label{f:E_3(D)_first}
    a^2 |A_s|^2 \le \E(A,A_s) |A \pm A_s| \le \E^{1/2}_3 |A_s| |A \pm A_s| \,.
\end{equation}
    Multiplying the last inequality by $|A_s|$, summing over $s\neq 0$
    and using the assumption
    $\E_3 (A) \ge 2|A|^3$,
    we get
\begin{equation}\label{tmp:22.01.2014_1}
    a^2 \E^{1/2}_3 \ll \o \E \,.
\end{equation}
    On the other hand, by Lemma \ref{corpop}, we have
\begin{equation}\label{tmp:22.01.2014_2}
    a^6 \ll \E_3 \sum_{s\in A-A,\, s\neq 0} |A \pm A_s| \le \E_3 |A-A| \o \,.
\end{equation}
    Combining (\ref{tmp:22.01.2014_1}) with (\ref{tmp:22.01.2014_2}), we obtain
$$
    \o^3 \gg \frac{a^{10}}{|A-A| \E^2}
$$
    as required.

    Now let us obtain (\ref{f:E_3(D)_0'}).
    Using Lemma \ref{l:eigen_A'}, we find $A'$, $|A'| \ge |A|/2$ such that estimate (\ref{f:eigen_A'}) takes place.
    As in (\ref{f:E_3(D)_first}), we get
\begin{equation}\label{f:E_3(D)_first'-}
    a^2 |A'_s|^2 \le \E(A,A'_s) |A \pm A'_s| \le \frac{2\E}{a} |A'_s| |A \pm A_s| \,.
\end{equation}
    By assumption $\E_3 (A) \ge 2^4 \gamma^{-1} |A|^3$.
    Using the connectedness, we obtain
\begin{equation}\label{tmp:09.04.2014_I''}
    \E_3 (A') \ge \gamma \frac{|A'|^6}{|A|^6} \E_3 (A) \ge 2 |A'|^3 \,.
\end{equation}
    Multiplying
    inequality (\ref{f:E_3(D)_first'-})
    by $|A'_s|$, summing over $s \neq 0$, we have in view of (\ref{tmp:09.04.2014_I''}) that
\begin{equation}\label{tmp:22.01.2014_3}
    a^2 \E_3 (A') \ll a^{-1} \E \E^* (A') \o \,.
\end{equation}
    Combining the last formula with first inequality from (\ref{tmp:09.04.2014_I''}), we get
\begin{equation}\label{tmp:09.04.2014_I'}
    a^2 \gamma \E_3 (A) \ll a^{-1} \E \E^* (A') \o \,.
\end{equation}
    On the other hand, summing the first estimate from  (\ref{f:E_3(D)_first'-}) over $s\neq 0$
    and applying Lemma \ref{l:E_3_A_s}, we see that
\begin{equation}\label{tmp:09.04.2014_I'''}
    a^2 \E^* (A') \le \o \E_3 \,.
\end{equation}
    Combining (\ref{tmp:09.04.2014_I'}), (\ref{tmp:09.04.2014_I'''}), we obtain
$$
    \o^2 \gg \gamma \frac{a^5}{\E} \,.
$$
    This completes the proof.
$\hfill\Box$
\end{proof}

\bigskip

From (\ref{f:E_3(D)_0}) it follows that if $|A-A| = K|A|$, $\E(A) \ll |A|^3 /K$, $\E_3 (A) \ge 2|A|^3$
then there is
$s\neq 0$ such that $|A-A_s| \gg K^{1/3} |A|$ as well as there exists
$s\neq 0$ with $|A+A_s| \gg K^{1/3} |A|$.
It improves a trivial lower bound $|A\pm A_s| \ge |A|$.
Using bound
(\ref{f:E_3(D)_0'})
one can show that there is $s\neq 0$ such that $|A-A_s| \gg K^{1/2} |A|$
(here $\E(A) \ll |A|^3 /K$),
provided by some connectedness assumptions take place.

We need in lower bounds on $\E_k (A)$ in Theorem \ref{t:E_3(D)-} and Proposition \ref{p:f:E_4(D,A)} below
to be separated from a very natural simple example, which can be called a "random sumset"\, case.
Namely, take a random $A\subseteq \Gr$ and consider $A\pm A$.
This "random sumset"\, has almost no structure (provided by $A\pm A$ is not a whole group, of course)
and we cannot say something useful in the situation.
It does not contradict Theorem \ref{t:E_3(D)-} and Proposition \ref{p:f:E_4(D,A)} (see also the results of the next section)
because the energies $\E_k (A)$ are really small in the case.

\bigskip


Now we give another prove of estimate (\ref{f:E_3(D)_0'}) which can be derived from
inequality (\ref{f:E_4(D,A)}), case $k=2$ below.
Actually, it gives us even stronger inequality, namely, $|A^2 - \D (A_s)| \gg |A|^5 \E^{-1} (A)$
for some $s\neq 0$, or, more generally, (see formulas (\ref{tmp:06.05.2014_1}), (\ref{tmp:06.05.2014_2}) below)
\begin{equation}\label{f:max_s,k}
    \max_{s\neq 0} |A^k \pm \D (A_s)|
        \gg
            \max_{r\ge 1} \left\{ \frac{|A|^{2k+1} \E_{r+1} (A)}{\E_r (A) \E_{k+1} (A)} \right\}
    \,,
\end{equation}
provided by some connectedness assumptions take place.
In particular, taking $r=k$ and $r=1$ in the previous formula, we get
$$
    \max_{s\neq 0} |A^k \pm \D (A_s)|
        \gg
            \left\{ \frac{|A|^{2k+1}}{\E_{k} (A)} \,, \frac{|A|^{2k-1} \E(A)}{\E_{k+1} (A)} \right\} \,.
$$

\begin{proposition}
    Let $A\subseteq \Gr$ be a set, $k\ge 2$ be a positive integer.
    Take two sets $D,S$ such that $A-A \subseteq D$, $A+A \subseteq S$.
    Then
\begin{equation}\label{f:E_4(D,A)}
    \gamma |A|^{2k+1} \ll_k \sum_{x\neq 0} (A\c A)^k (x) d_k (x) \le \sum_{x\neq 0} (A\c A)^k (x) (D\c D)^k (x) \,,
\end{equation}
\begin{equation}\label{f:E_4(D,A)'}
    \gamma |A|^{2k+1} \ll_k \sum_{x\neq 0} (A\c A)^k (x) s_k (x) \le \sum_{x\neq 0} (A\c A)^k (x) (S\c S)^k (x) \,,
\end{equation}
    provided by $A$ is $(k+1,\beta,\gamma)$--connected with $\beta \le 1/2$
    and
    $\E_{k+1} (A) \ge 2^{k+2} \gamma^{-1} |A|^{k+1}$.
    Here $d_k (x) = \sum_\a D_x (\a) (D \c D_x)^{k-1} (\a)$, $s_k (x) = \sum_\a S_x(\a) (S_x * D)^{k-1} (\a)$.
\label{p:f:E_4(D,A)}
\end{proposition}
\begin{proof}
    Using Lemma \ref{l:eigen_A'}, we find $A'$, $|A'| \ge |A|/2$ such that estimate (\ref{f:eigen_A''}) takes place
    with $g(z) = (A \c A)^k (z)$.
    It follows that
\begin{equation}\label{tmp:06.05.2014_1}
    |A|^{2k} |A'_x|^2 \le |A^k \pm \D(A'_x)| \E_{k+1} (A'_x, A)
        \le |A^k \pm \D(A_x)| \frac{2\E_{k+1} (A)}{|A|} |A_x| \,.
\end{equation}
    Multiplying the last inequality by $|A'_x|^{k-1}$ and
    summing
    over $x\neq 0$ (to obtain (\ref{f:max_s,k}) multiply by $|A'_x|^{r-1}$), we get
\begin{equation}\label{tmp:06.05.2014_2}
    \gamma |A|^{2k+1} \E_{k+1} (A) \ll |A|^{2k+1} \E^*_{k+1} (A')
        \ll \E_{k+1} (A) \sum_{x \neq 0} |A^k \pm \D(A_x)| |A_x|^k \,.
\end{equation}
    Here we have used  the fact
$$
    \E_{k+1} (A') \ge \gamma \frac{|A'|^{2(k+1)}}{|A|^{2(k+1)}} \E_{k+1} (A) \ge 2 |A'|^{k+1}
$$
    and the assumption $\E_{k+1} (A) \ge 2^{k+2} \gamma^{-1} |A|^{k+1}$.
    Note that by Katz--Koester trick (\ref{f:KK_trick_new}), one has $|A^k - \D(A_x)| \le d_k (x)$, $|A^k + \D(A_x)| \le s_k (x)$
    (or just see the proof of Proposition \ref{pr:E_k(D)_simple} below).
    Finally,  $d_k (x) \le |D_x|^k$, $s_k (x) \le |S_x|^k$
    and  we obtain (\ref{f:E_4(D,A)}).
    This completes the proof.
$\hfill\Box$
\end{proof}

\bigskip

Using Katz--Koester trick or just the Cauchy--Schwarz inequality
one can show that (\ref{f:E_4(D,A)}) trivially takes place in the case $k=1$
without any conditions on connectedness of $A$ or lower bounds for any sort of energy.
Under the assumptions of Proposition \ref{p:f:E_4(D,A)} from (\ref{f:E_4(D,A)}) it follows that
$$
    \gamma^2 |A|^{4k+2} \ll_k \E^*_{2k} (A) \E^*_{2k} (D)
$$
and similarly for $S$.
Some weaker results
but without any conditions on $A$
were obtained in \cite{SS1}.


\bigskip

Results above give an interesting
corollary on non--random sumsets/difference sets.
Put $D=A-A$, $S=A+A$.
Then for an arbitrary positive integer $k$ and any elements $a_1,\dots,a_k\in A$, we have
\begin{equation}\label{f:A,D_inclusion}
    A \subseteq (D + a_1) \bigcap (D + a_2) \dots \bigcap (D+a_k) \,, \quad
    A \subseteq (S - a_1) \bigcap (S - a_2) \dots \bigcap (S-a_k)
    \,.
\end{equation}
In particular, there is $x\in D$, $x\neq 0$ such that $|D_x|, |S_x| \ge |A|$
(also it follows  from Katz--Koester inclusion (\ref{f:KK_trick_new})).
By Corollary \ref{c:A^2-D(A)} (see also Lemma \ref{l:A^2_pm}) one can improve it to
$|D_x|, |S_x| \ge K^\eps |A|$,
where $K=|A|^3 \E^{-1} (A)$.
Theorem \ref{t:E_3(D)-} gives us
even stronger result.

\begin{corollary}
Let $A \subseteq \Gr$ be a set, $D=A-A$, $S=A+A$, $\E (A) = |A|^3/K$,
Suppose that $A$ is $(3,\beta,\gamma)$--connected, $\beta \le 1/2$,
and $\E_3 (A) \ge \gamma^{-1} 2^{4} |A|^3$.
Then there is $x\neq 0$ such that $|D_x|,|S_x| \gg \gamma^{1/2} K^{1/2} |A|$.
\label{c:Dx}
\end{corollary}

One can
get
an analog of Corollary \ref{c:Dx} for multiple intersections (\ref{f:A,D_inclusion})
but another types of energies will require in the case.
Nevertheless, some weaker inequality of the form $|D_{\v{x}}| \ge K^{\eps} |A|$ can be obtained,
using Proposition \ref{p:E_3_and_E_critical} and
Theorem \ref{t:E_3_and_E_critical}.
Here $K=|A|^{k+1} \E^{-1}_k (A)$, $k\ge 2$.
Interestingly, we do not even need in any connectedness in this weaker result.

\begin{proposition}
    Let $A\subseteq \Gr$ be a set,  $k\ge 2$ be a positive integer, $c\in (0,1]$ be a real number,
    $\E_{k} (A) \ge k^2 \E_{k-1} (A)$,
    and $K=|A|^{k+1} \E^{-1}_k (A) \le |A|^{1-c}$, $K_1 = |A|^3 \E^{-1} (A)$.
    Then for all sufficiently small $\eps = \eps (c) >0$
    there is $\v{x} = (x_1,\dots,x_{k-1})$ with distinct $x_j$, $j\in [k]$ such that
\begin{equation}\label{}
    |D_{\v{x}}|\,, |S_{\v{x}}| \gg_k |A| \cdot \min\{ K^{\eps}, c_{K^{\eps}} K_1 \} \,,
\end{equation}
    where the constant $c_{K^{\eps}}$ satisfies $c_{K^{\eps}} \gg_{K^{\eps}} 1$
    and, again, the degree of the polynomial dependence is a function on $c$.
\label{p:DxSx}
\end{proposition}
\begin{proof}
Suppose not.
Take any set $\mathcal{P} \subseteq A^{k-1} - \D(A)$.
Applying Lemma \ref{l:E_3_A_s}, one has
$$
    |A|^2 \left( \sum_{\v{x} \in \mathcal{P}} |A_{\v{x}}| \right)^2
        =
    \left( \sum_{\v{x} \in \mathcal{P}} \sum_z (A\c A_{\v{x}}) (z) \right)^2
        \le
            \E_{k+1} (A) \sum_{\v{x} \in \mathcal{P}} |A \pm A_{\v{x}}|
                \le
$$
\begin{equation}\label{tmp:09.04.2014_1}
                \le
                    \E_{k+1} (A) |\mathcal{P}| \cdot \max_{\v{x} \in \mathcal{P}} |A \pm A_{\v{x}}| \,.
\end{equation}
Note that the assumption $\E_{k} (A) \ge k^2 \E_{k-1} (A)$ implies
\begin{equation}\label{tmp:09.04.2014_2}
    \E_k (A) = \sum_{\v{x}} |A_{\v{x}}|^2 \le \sum'_{\v{x}} |A_{\v{x}}|^2 + \binom{k-1}{2} \E_{k-1} (A)
        \le
            2 \sum'_{\v{x}} |A_{\v{x}}|^2 \,,
\end{equation}
where the sum $\sum'$ above is taken over $\v{x} = (x_1,\dots,x_{k-1})$ with distinct $x_j$.
Now take $\mathcal{P}$ such that  $\D < |A_{\v{x}}| \le 2 \D$ for $\v{x} = (x_1,\dots,x_{k-1}) \in \mathcal{P}$,
where all $x_j$, $j\in [k]$ are distinct
and
$$
    \sum_{\v{x} \in \mathcal{P}} |A_{\v{x}}|^2 \gg \frac{\E_{k} (A)}{\log K} \,.
$$
Of course, such $\mathcal{P}$ exists by the pigeonhole principle and bound (\ref{tmp:09.04.2014_2}).
Using the last inequality, and recalling (\ref{tmp:09.04.2014_1}), we obtain
$$
    \frac{|A|^2 \E_{k} (A)}{\log K}
        \ll
            |A|^2 |\mathcal{P}| \D^2
        \ll
            \E_{k+1} (A) K^{\eps} |A| \,.
$$
In other words, $\E_{k+1} (A) \gg_{K^\eps} |A| \E_k (A)$.
Put $M=\max\{ K^\eps, k\}$.
Applying Proposition \ref{p:E_3_and_E_critical} as well as
Theorem \ref{t:E_3_and_E_critical},
we see that
$A \approx_{M} \L \dotplus H$.
After that put $A' = A\cap (H+\L)$.
Then $|A'| \gg_{M} |A|$ and $|H| \gg_M |A| /K$.
Further, as in the proof of Theorem \ref{t:H+L_description}, we get
$$
    \sum_{\v{x},\,\|x\|=k-1} \Cf_{k} (A') (\v{x}) \Cf_{k} (H) (\v{x})
        =
            \sum_s (A' \c H)^k (s)
    \ge \frac{|H|^{k} |A'|^{k}}{|A'-H|^{k-1}} \gg_{M} \frac{|H|^k |A|^k}{|\L+H-H|^{k-1}}
$$
$$
    \gg_{M} \frac{|H|^k |A|^k}{(|\L||H-H|)^{k-1}}
        \gg_{M}
            |A| |H|^k
$$
and hence
\begin{equation}\label{tmp:09.04.2014_3}
    \sum_{\v{x},\,\|x\|=k-1} |A'_{\v{x}}| |H_{\v{x}}| \gg_{M} |A| |H|^k \,.
\end{equation}
In particular, there are at least $\gg_{M} |H|^{k-1}$ elements
$\v{x} \in H^{k-1}- \D(H)$
such that
$|A'_{\v{x}}| \gg_{M} |A|$.
We can suppose that the summation in (\ref{tmp:09.04.2014_3}) is taken over
$\v{x} = (x_1,\dots,x_{k-1})$ with distinct $x_j$ because of the rest is bounded by
$$
    \binom{k-1}{2} \sum_{x} (A\c H)^{k-1} (x) \le \binom{k-1}{2} |H|^{k-1} |A| \ll_{M} |A| |H|^k \,.
$$
The last estimate follows from the assumption $K \le |A|^{1-c}$.
Choosing any such $\v{x}$ and using the Cauchy--Schwarz inequality, we obtain
$$
    |A \pm A_{\v{x}}| \ge |A' \pm A'_{\v{x}}| \ge \frac{|A'|^2 |A'_{\v{x}}|^2}{\E(A',A'_{\v{x}})}
        \gg_{M}
            \frac{|A|^4}{\E(A)} = K_1 |A|
$$
and in view of Katz--Koester trick (\ref{f:KK_trick_new}),
we see that $|D_{\v{x}}|\,, |S_{\v{x}}|$ are huge for large $K_1$.
This concludes the proof.
$\hfill\Box$
\end{proof}

\bigskip

Using Proposition \ref{p:DxSx} one can derive an interesting dichotomy.

\begin{theorem}
    Let $A\subseteq \Gr$ be a set, $D=A-A$, $S=A+A$, $k\ge 2$, and
    $M\ge 1$, $\eps \in (0,1)$ be real numbers.
    Put $K=|A|^{k+1} \E^{-1}_k (A)$.
    Suppose that for any vector $\v{x} = (x_1,\dots,x_{k-1})$ with distinct $x_j$, $j\in [k]$
    the following holds
    \begin{equation*}\label{}
        |D_{\v{x}}| \le M|A| \quad \mbox{or, similarly, } \quad  |S_{\v{x}}| \le M|A| \,.
    \end{equation*}
    Then either $\E_k (A) \ll_{M,\, |A|^\eps,\, k} |A|^k$ or $\E(A) \gg_{M,\, |A|^\eps,\, k} |A|^3$.
    Again, the degree of the polynomial dependence is a function on $\eps$.
\label{t:dichotomy_DS}
\end{theorem}
\begin{proof}
    Put $K_1 = |A|^3 \E^{-1} (A)$.
    Suppose, in contrary, that $\E_k (A) \gg_{M,\, |A|^\eps,\, k} |A|^k$ and\\ $\E(A) \ll_{M,\, |A|^\eps,\, k} |A|^3$.
    Then $K \ll_{M,\, |A|^\eps,\, k} |A|^{}$ and $K_1 \gg_{M,\, |A|^\eps,\, k} 1$.
    Trivially, $\E_k (A) \le |A|^{k-2} \E(A)$
    and hence
    $$|A| \gg_{M,\, |A|^\eps,\, k} K \ge K_1 \gg_{M,\, |A|^\eps,\, k} 1 \,.$$
    Finally, by the  upper bound for the parameter $K$ the number $c$
    which is defined as $K = |A|^{1-c}$ can by taken depends on $\eps$ only.
    Thus,
    everything follows from Proposition \ref{p:DxSx}, the only thing we need to consider is the situation when
    $\E_k (A) \le k^2 \E_{k-1} (A)$.
    But in the case
$$
    |A|^k \le \E_k (A) \le k^2 \E_{k-1} (A) \le k^2 |A|^{k-3} \E(A) \,.
$$
    In other words $\E(A) \gg_k |A|^3$ and this concludes the proof.
$\hfill\Box$
\end{proof}


\bigskip

Thus, if $|D_{\v{x}}|\,, |S_{\v{x}}|$ are not much larger than $|A|$
then either $A$ is close to what we called a "random sumset"\, or, on the contrary, is very structured.
Clearly, the both situations are realized.

\section{Energies of sumsets}
\label{sec:sumsets1}

Let $A \subseteq \Gr$ be a set.
Throughout the section we put $D=A-A$ and $S=A+A$.
As was explained in the introduction that  one can hope to prove a good lower bound for $\E_3 (D)$.
It will be done in Theorem \ref{t:E_3(D)} below but before this we formulate a simple preliminary
lower bound for $\E^D_k (D)$, $\E^D_k (S)$.
Similar
lower bounds for $\E^D_2 (D)$, $\E^D_2 (S)$
were given
in Corollary 5.6 of paper \cite{M_R-N_S}.
Further, it was proved in \cite{SS1} (see Remark 8) that $\sigma_{k+1} (D) \ge |A^k - \D(A)|$.
Now we obtain a similar lower bound for $\E^{D}_k (D)$.
Recall that by $\E^{D}_1 (D)$ we mean $\sigma_{3} (D) = \sigma_D (D)$, that is $\sum_{x\in D} (D \c D) (x)$.

\begin{proposition}
    Let $A\subseteq \Gr$ be a set.
    Put $D=A-A$, $S=A+A$.
    Then for all $k\ge 1$ one has
    \begin{equation}\label{f:E_k(D)_simple}
        \E^D_k (D) \ge |A^{k+1} - \D(A)|
            \ge |A-A| |A|^k \,,
    \end{equation}
    and, similarly,
    $$
        \sum_{x} S(x) (S * D)^{k} (x) \ge |A^{k+1} + \D(A)| \ge |A|^k \max\{ |A-A|, |A+A| \} \,,
    $$
    \begin{equation}\label{f:E_k(D)_simple'}
        \E^D_k (S) \ge |A|^{k-1} |A^2 + \D(A)| \ge |A|^k \max\{ |A-A|, |A+A| \} \,.
    \end{equation}
\label{pr:E_k(D)_simple}
\end{proposition}
\begin{proof}
    The second estimates in (\ref{f:E_k(D)_simple}), (\ref{f:E_k(D)_simple'})
     follow from Lemma \ref{l:A^2_pm}.
    Further, it is easy to get (or see e.g. \cite{SS1}) that
$$
    |A^{k+1} - \D(A)| \le \sum_{x_1,\dots,x_{k+1}} D(x_1) \dots D(x_{k+1}) \prod_{i\neq j} D(x_i - x_j)
        \le
$$
$$
        \le
            \sum_{x_1,\dots,x_{k+1}} D(x_1) \dots D(x_{k+1}) D(x_1-x_2) \dots D(x_1-x_{k+1})
                =
                    \sum_{x} D(x) (D\c D)^{k} (x)
                        =
                            \E^D_{k} (D)
$$
as required.
Similarly,
$$
            |A^{k+1} + \D(A)| \le \sum_{x_1,\dots,x_{k+1}} S(x_1) \dots S(x_{k+1}) \prod_{i\neq j} D(x_i - x_j)
        \le
$$
$$
        \le
            \sum_{x_1,\dots,x_{k+1}} S(x_1) \dots S(x_{k+1}) D(x_1-x_2) \dots D(x_1-x_{k+1})
                =
                    \sum_{x} S(x) (S* D)^{k} (x)
                            \,.
$$
Finally, by Lemma \ref{l:A^2_pm}, we get
$$
    \E^D_k (S) = \sum_{s\in D} (S\c S)^k (s) \ge \sum_{s\in D} |A+A_s|^k
        \ge |A|^{k-1} \sum_{s\in D} |A+A_s|
            =
$$
$$
            =
                |A|^{k-1} |A^2 + \D(A)| \ge |A|^k \max\{ |A-A|, |A+A| \} \,.
$$
This completes the proof.
$\hfill\Box$
\end{proof}

\bigskip

Interestingly, that some sort of sumset, namely, $A^n \pm \D(A)$
gives a lower bound for an energy, although, usually, the energy
provides lower bounds for the cardinality  sumsets via the Cauchy--Schwarz inequality.
The trick allows to obtain a series of results in \cite{SS1}---\cite{SS3}.
Although bounds (\ref{f:E_k(D)_simple}), (\ref{f:E_k(D)_simple'}) are very simple they can be tight in some cases.
For example, take $A$ to be a dissociated set
or, in contrary, a very structural set as a subspace.

\bigskip

Now we formulate the main result of the section concerning lower bounds for some energies of sumsets/difference sets.
Again we need in lower bounds on $\E_k (A)$ in Theorem \ref{t:E_3(D)} below
to be separated from the "random sumset"\, case.

\begin{theorem}
    Let $A\subseteq \Gr$ be a set.
    Take two sets $D,S$ such that $D=A-A$, $S=A+A$.
    Then
\begin{equation}\label{f:E_3(D)_1}
    \E^2_3 (D,A,A),\,~ \E^2_3 (S,A,A)
        \ge \frac{|A|^{13}}{|A-A|^2 \E (A)} \,,
\end{equation}
    and
\begin{equation}\label{f:E_3(D)_2}
    (\E^D_3 (D))^4 \ge \max \left\{ |D|^{12}, \frac{|A|^{45}}{\E^9 (A) |D|^2} \right\} \,,
        \quad \quad
    (\E^D_3 (S))^4 \ge \max \left\{ |S|^{12}, \frac{|A|^{45}}{\E^9 (A) |D|^2} \right\} \,.
\end{equation}
    Further, let $\beta,\gamma \in [0,1]$ be real numbers,  $\beta \le 1/2$.
    If $A$ is $(2,\beta,\gamma)$--connected then
\begin{equation}\label{f:E_3(D)_2.5}
    \E^2_3 (D,A,A),\,~ \E^2_3 (S,A,A)
        \gg
            \gamma |A|^5 \E (A) \,.
\end{equation}
    Suppose that $A$ is $(3/2,\beta,\gamma)$ and $(2,\beta,\gamma)$--connected, correspondingly.
    Then
\begin{equation}\label{f:E_3(D)_3}
    \E^D_3 (D),\, \E^D_3 (S)
        \gg
            \gamma \frac{|A|^{33/4} \E_{3/2} (A)}{\E^{9/4} (A) \log |A|} \,,
        \quad \quad \quad
    \E^D_3 (D),\, \E^D_3 (S)
                \gg
                    \gamma \frac{|A|^{17/2}}{\E^{3/2} (A) \log |A|} \,,
\end{equation}
correspondingly.
\label{t:E_3(D)}
\end{theorem}
\begin{proof}
    Write $\E = \E(A)$, $\E_3 = \E_3 (A)$, and $a=|A|$.
    Let us obtain bounds (\ref{f:E_3(D)_1}), (\ref{f:E_3(D)_2.5}).
    Using Lemma \ref{l:eigen_A'}, we find $A'$, $|A'| \ge |A|/2$ such that estimate (\ref{f:eigen_A'}) takes place.
    As in the proof of inequalities (\ref{f:E_3(D)_first}), (\ref{f:E_3(D)_first'-}), we get
\begin{equation}\label{f:E_3(D)_first'}
    a^2 |A'_s|^2 \le \E(A,A'_s) |A \pm A'_s| \le \frac{2\E}{a} |A'_s| |A \pm A_s| \,.
\end{equation}
    Multiplying the last inequality by $|A'_s|$, summing over $s$ and using Katz--Koester trick, we have
\begin{equation}\label{tmp:22.01.2014_3}
    a^2 \E_3 (A') \ll a^{-1} \E \cdot \E_3 (D,A)
\end{equation}
    and similar for $S$.
    On the other hand by the second part of Lemma \ref{corpop}, we obtain
\begin{equation}\label{tmp:22.01.2014_4}
    \left( \frac{a^4}{|A-A|} \right)^2 a^2 \ll \E^2 (A') a^2 \ll \E_3 (A') \cdot \E_3 (D,A)
\end{equation}
    and using the first part of the lemma, we have the same bound for $S$
\begin{equation}\label{tmp:22.01.2014_4+}
    \left( \frac{a^2}{|A-A|} \right)^2 a^6 \ll \sigma^2_{\t{D}} (A') (a')^2 \left( \frac{(a')^2}{2|D|} \right)^2
            \le \E_3 (A') \sum_{x \in \t{D}} |S_x| |A_x|^2
        \le \E_3 (A') \cdot \E_3 (S,A) \,,
\end{equation}
    where $\t{D} := \{ x\in D ~:~ |A'_x| \ge (a')^2 /2|D| \}$, $\sigma_{\t{D}} (A') \ge (a')^2 /2$
    (for details, see \cite{SS3}).
    Another way to prove the same is to use Lemma \ref{l:T_A,B} with $A=B=A'$, $\psi (x) = (A'\c A') (x)$.
    Combining (\ref{tmp:22.01.2014_3}) and (\ref{tmp:22.01.2014_4}), (\ref{tmp:22.01.2014_4+}), we get
$$
    \E^2_3 (D,A)\,,~ \E^2_3 (S,A)
        \gg \frac{|A|^{13}}{|A-A|^2 \E (A)}
$$
    Using the tensor trick (see \cite{TV} or \cite{s_mixed}), we have (\ref{f:E_3(D)_1}).
    If $A$ is $(2,\beta,\gamma)$--connected then
    $$
        \E(A') \gg \gamma \E(A)
    $$
    and combining the last inequality with (\ref{tmp:22.01.2014_3})
    and the second bound from (\ref{tmp:22.01.2014_4}),
    we get (\ref{f:E_3(D)_2.5})
    (to obtain lower bound for $\E^2_3 (S,A)$
    one should use
    Lemma \ref{l:T_A,B}).

    It remains to prove (\ref{f:E_3(D)_2}) and (\ref{f:E_3(D)_3}).
    Returning to (\ref{f:E_3(D)_first'})---(\ref{tmp:22.01.2014_3}), we obtain
$$
    a^3 \E_3 (A') \ll \E \sum_s |A'_s|^2 |A' \pm A'_s|
$$
    or, by the H\"{o}lder inequality
\begin{equation}\label{tmp:22.01.2014_5}
    a^9 \E_3 (A') \ll \E^3 \E^D_3 (D) \,.
\end{equation}
    On the other hand, by Lemma \ref{corpop} for any $P\subseteq A'-A'$, we have
$$
    a^2 \sigma^2_P (A') \ll \E_3 (A') \sum_{s\in P} |A' \pm A'_s| \,.
$$
    Applying the H\"{o}lder inequality, we get
\begin{equation}\label{tmp:22.01.2014_5'}
    a^6 \sigma^6_P (A') \ll \E^3_3 (A') |P|^2 \E^D_3 (D)
\end{equation}
    and similarly for $S$.
    Now choose $P\subseteq A'-A'$ such that $P = \{ s\in A'-A' ~:~ \rho < |A'_s| \le 2 \rho \}$
    for some positive number $\rho$ and such that
$$
    \sum_{s\in P} |A'_s|^{3/2} \gg \frac{\E_{3/2} (A')}{\log a} \,.
$$
    Of course, the set $P$ exists by Dirichlet principle.
    Combining the last inequality with (\ref{tmp:22.01.2014_5'}), we obtain
\begin{equation}\label{tmp:22.01.2014_6}
     a^6 \E^4_{3/2} (A') \log^{-4} a \ll a^6 ( |P| \rho^{3/2} )^4 \ll \E^3_3 (A') \E^D_3 (D)
     \,.
\end{equation}
    Using estimates (\ref{tmp:22.01.2014_5}), (\ref{tmp:22.01.2014_6}), we have
\begin{equation}\label{tmp:22.01.2014_7}
    (\E^D_3 (D))^4 \E^9 \gg a^{33} \E^4_{3/2} (A') \log^{-4} a \,.
\end{equation}
    Thus
$$
    (\E^D_3 (D))^4 \,,~ (\E^D_3 (S))^4 \gg \frac{a^{45}}{\E^9 |D|^2 \log^4 a} \,.
$$
    Applying the tensor trick again, we get (\ref{f:E_3(D)_2}).
    To obtain (\ref{f:E_3(D)_3}) recall that $A$ is $(3/2,\beta,\gamma)$--connected set.
    Hence by (\ref{tmp:22.01.2014_7}), we obtain
$$
    (\E^D_3 (D))^4 \E^9 \gg \gamma^4 a^{33} \E^4_{3/2} (A) \log^{-4} a
$$
    and the first formula of (\ref{f:E_3(D)_3}) follows.
    Further, because of $A$ is $(2,\beta,\gamma)$--connected set then
    using Lemma \ref{l:connected} for $A'$ as well as (\ref{tmp:22.01.2014_7}), we have
$$
    (\E^D_3 (D))^2 \E^3 \gg \gamma^2 a^{17} \log^{-2} a
$$
    and the last estimate coincide with the second inequality in (\ref{f:E_3(D)_3}).
    This completes the proof.
$\hfill\Box$
\end{proof}

\bigskip

From (\ref{f:E_3(D)_2}), the definition of the number $K$ as $|D|=K|A|$
and the assumption $\E(A) \ll |A|^3 /K$, we get
\begin{equation}\label{f:E_3_74}
    \E_3 (D) \gg K^{7/4} |A|^4 \,.
\end{equation}
An upper bound here is $K^{2} |A|^4$ and it follows from the main example of section \ref{sec:structural},
that is $\Gr = \f_2^n$, $A=H\dotplus \L$.
Note also that the second inequality in formula (\ref{f:E_3(D)_3}) is weak but do not depends on the size of $A-A$
or on the energy $\E_{3/2} (A)$.

\bigskip

As
for dual quantities $\T_k (D)$, $\T_k (S)$,
our example $A=H\dotplus \L$ shows that
there are not nontrivial lower bounds for $\T_k (A\pm A)$ in general,
which is better than
a simple consequence of Katz--Koester
$$
    \T_k (D) \ge |A|^2 \T_{k-1} (D) \ge \dots \ge |A|^{2(k-2)} \E (D) \ge |A|^{2k-2} |D| \,,
$$
$$
    \T_k (S) \ge |A|^2 \T_{k-1} (S,\dots,S,D) \ge \dots \ge |A|^{2(k-2)} \E (D) \ge |A|^{2k-2} |D|
$$
(just use $(D\c D)(x) \ge |A| D(x)$ and $(S\c S) (x) \ge |A| D(x)$).
The reason is that the structure of $A\pm A$ is similar
to the structure of $A$ in the case, of course.
Nevertheless,  it was proved in \cite{SS1}, Lemma 3 that
\begin{equation}\label{f:Lev-}
    |A|^{4k} \le \E_{2k} (A) \T_k (A\pm A) \,,
\end{equation}
and also in \cite{s_mixed}, Note 6.6 that
\begin{equation}\label{f:Lev}
    \left( \frac{\sum_{x\in P} (A\c A) (x)}{|A|} \right)^{4k} \le \E_{2k} (A) \T_k (P) \,,
\end{equation}
where $P\subseteq D$ is any set.
So, if we know something on $\E_{2k} (A)$ then it gives us a new information about $\T_k (A\pm A)$.
Trivially, formula (\ref{f:Lev-}) implies that $\T_k (A\pm A) \ge |A|^{2k+2} / \E(A)$.
Again, the last inequality is sharp as our main example $A=H\dotplus \L$ shows.

\bigskip

Vsevolod F. Lev asked the author about an analog of (\ref{f:Lev}) for different sets $A$ and $B$.
Proposition \ref{p:Lev_question} below is our result in the direction.
The proof is in spirit of \cite{s_mixed}.
For simplicity we consider the case $k=2$ only.
The case of greater powers of two is considered similarly if one take $M^2,M^4,\dots$
or just see the proof of Theorem 6.3 from \cite{s_mixed} (the case of any $k$).
We do not insert the full proof because we avoid to use the operators from \cite{s_ineq}, \cite{s_mixed}
in the paper which is considered to be elementary.

\begin{proposition}
    Let $k\ge 2$ be a power of two, $A,B\subseteq \Gr$ be two sets, and $P\subseteq A-B$.
    Then
$$
    \left( \frac{\sum_{x\in P} (A \c B) (x)}{|A|^{1/2} |B|^{1/2}} \right)^{4k}
        \le
            \E_{2k} (A,\dots,A,B,\dots,B) \T_k (P) \,.
$$
\label{p:Lev_question}
\end{proposition}
\begin{proof}
    Let $k=2$.
    Define
    the matrix
$$
    M(x,y) = P(x-y) A(x) B(y)
$$
    and calculate its rectangular norm
$$
    \la^4_1 (M) \le \sum_j \la^4_j (M)
        =
    \sum_{x,y,x',y'} M(x,y) M(x',y) M(x,y') M(x',y')
        =
$$
$$
        =
            \sum_{x,x'\in A}\, \sum_{y,y'\in B} P(x-y) P(x'-y) P(x-y') P(x'-y')
        =
$$
$$
    =
        \sum_{\a,\beta,\gamma} \Cf_4 (B,A,A,B) (\a,\beta,\gamma) P(\a) P(\beta) P(\a-\gamma) P(\beta-\gamma) \,,
$$
where $\la_j (M)$ the singular numbers of $M$.
Clearly,
$$
    \la_1 (M) \ge \frac{\sum_{x\in P} (A \c B) (x)}{|A|^{1/2} |B|^{1/2}} \,.
$$
Thus, by the Cauchy--Schwarz inequality, we get
$$
    \left( \frac{\sum_{x\in P} (A \c B) (x)}{|A|^{1/2} |B|^{1/2}} \right)^8
        \le
            \E_4 (A,A,B,B) \E (P)
$$
    as required.
    This completes the proof.
$\hfill\Box$
\end{proof}

\bigskip

We end this section showing that there is a different way to prove our Theorem \ref{t:E_3(D)} using
slightly bigger sets $D_x$ or $S_x$ not $A\pm A_x$.
%
%
The proof based on a lemma,
which demonstrates,
in particular, that $A\pm A$ contains approximately $|A|^3 \E^{-1} (A)$
almost disjoint translates of $A$, roughly.
In the proof we use arguments from \cite{ALON}.

\begin{lemma}
    Let $A,B \subseteq \Gr$ be two sets.
    Then there are
    \begin{equation}\label{f:s_bound}
        s \ge 2^{-4} |A| |B|^2 \E^{-1} (A,B)
    \end{equation}
    disjoint sets $A_j \subseteq A+b_j$, $|A_j|\ge |A|/2$, $b_j \in B$, $j\in [s]$.

Moreover, for any set $S\subseteq A+B$ put $\sigma = \sum_{x\in S} (A * B) (x)$.
Suppose that $\sigma \ge 16 |B|$.
Then there are
\begin{equation}\label{f:s_bound+}
    s \ge 2^{-8} \sigma^3 |A|^{-2} |B|^{-1} \E^{-1} (A,B)
\end{equation}
disjoint sets $S_j \subseteq S \cap (A+b_j)$,
$|S_j|\ge 2^{-3} \sigma |B|^{-1}$, $b_j \in B$, $j\in [s]$.
\label{l:sumsets_disjoint}
\end{lemma}
\begin{proof}
Let us begin with (\ref{f:s_bound}).
Put $S=A+B$.
Our arguments is a sort of an algorithm.
At the first step of the algorithm take $A_1 = A+b$, where $b\in B$ is any element of $B$.
Suppose that we have constructed disjoint sets $A_1,\dots,A_k$.
If there is $b\in B$ such that $|(A+b) \setminus \bigsqcup_{j=1}^k A_j| \ge |A|/2$
then put $A_{k+1} = (A+b) \setminus \bigsqcup_{j=1}^k A_j$ and take $b_{k+1} = b$.
Suppose that our algorithm stops after $s$ steps.
If $s\ge |B|/2$ then we are done.
Put $U = \bigsqcup_{j=1}^s A_j$ and $B_* = B\setminus \{ b_1,\dots,b_s \}$.
Then $s|A|/2 < |U| \le s |A| $ and $|B_*| \ge |B|/2$.
We have
$$
    2^{-2} |A| |B| \le 2^{-1} |A| |B_*| \le \sum_x (A \c U) (x) B_* (x) \,.
$$
Using the Cauchy--Schwarz inequality, we obtain
$$
    2^{-4} |A|^2 |B|^2 \le \E(A,B) |U| \le \E(A,B) s |A|
$$
and the required lower bound for $s$ follows.

Let us prove the second part of the lemma.
Put $a=|A|$, $b=|B|$.
First of all note that
$$
    \sigma = \sum_{x\in S} (A * B) (x) = \sum_{x\in B} |S \cap (A+x)|
$$
and hence there is $x\in B$ such that $|S\cap (A+x)| \ge \sigma b^{-1}$.
Put $b_1 = x$, and let $S_1 \subseteq S\cap (A+b_1)$ be an arbitrary set of size $\lceil \eps \sigma b^{-1} \rceil$,
where $\eps = 1/8$.
After that using the arguments as above, we construct a family of disjoint sets
$S_j \subseteq S \cap (A+b_j)$, $|S_j|\ge \eps \sigma b^{-1}$, $b_j \in B$, $j\in [s]$.
If $s \ge \sigma/(4a)$ then we are done.
If not then
put $U=\bigsqcup_{j=1}^s S_j$ and $B_* = B\setminus \{ b_1,\dots,b_s \}$.
We have $B_* \subseteq B'_* \bigcup B''_*$, where
$$B'_* = \{ x\in B_* ~:~ (A \c U) (x) \ge \eps \sigma b^{-1} \} \,,$$
$$B''_* = \{ x\in B_* ~:~ |(A+x) \setminus S| \ge a-2\eps \sigma b^{-1} \} \,.$$
Further, because of
$$
    (a-2\eps \sigma b^{-1}) |B''_*| \le \sum_{x\in B''_*} |(A+b) \setminus S|
        \le
            \sum_{x\in B} |(A+x) \setminus S|
                \le
                    ab - \sum_{x\in B} |(A+x) \bigcap S|
                        =
                            ab - \sigma
$$
we see that
$
    |B''_*| \le (b-\sigma/a)(1+4\eps \sigma/(a b)).
$
Thus,
$$
    |B'_*| \ge b-s -(b-\sigma/a)(1+4\eps \sigma/(a b))
        \ge
            \sigma/a - 4\eps \sigma/a - s
                \ge
                    \sigma/4a \,.
$$
Finally, we obtain
$$
    \eps \sigma b^{-1} \cdot \sigma/4a \le \sum_x (A \c U) (x) B'_* (x) \le \sum_x (A \c U) (x) B (x)
$$
and hence, in view of the condition $\sigma \ge 16 b$ the following holds
$$
    (\eps \sigma b^{-1})^2 \cdot (\sigma/4a)^2
        \le
            |U| \E(A,B)
                \le
                    (2 \eps \sigma b^{-1}) s \E(A,B) \,.
$$
Whence, $s\ge 2^{-8} \sigma^3 a^{-2} b^{-1} \E^{-1} (A,B)$.
This concludes the proof.
$\hfill\Box$
\end{proof}

\bigskip

Now let us show how to get (\ref{f:E_3(D)_2.5}), for example.
Applying Lemma \ref{l:sumsets_disjoint} with $A=A$, $B=-A$, we obtain
\begin{equation}\label{tmp:11.05.2014_1}
    \E_3 (D,A,A) \ge \sum_{j=1}^s \E_3 (A_j,A,A) \gg |A|^3 \E^{-1} (A) \E_3 (A)
\end{equation}
provided by $(3,\beta,\gamma)$--connectedness assumptions, $\beta,\gamma \gg 1$
(by the way bound (\ref{tmp:11.05.2014_1}) is tight as our main example $H\dotplus \Lambda$ shows).
On the other hand, we have by Lemma \ref{corpop} that
$$
    \E^2 (A) |A|^2 \le \E_3 (A) \E_3 (D,A,A) \,.
$$
Combining the last two bounds, we get (\ref{f:E_3(D)_2.5}).

\bigskip

Using similar arguments and the second part of Lemma \ref{l:sumsets_disjoint}, we obtain
the following consequence, which shows, in particular,
that the popular difference sets \cite{Gow_m}, \cite{TV}
have some structure in the sense that they have large
energy
of some sort.

\begin{corollary}
    Let $A\subseteq \Gr$ be a set, $P\subseteq A-A$.
    Suppose that $A$ is a $(3,\beta,\gamma)$--connected set,
    $\beta \le 2^{-3} \sigma_P (A) |A|^{-1}$.
    Then
\begin{equation}\label{f:E_3(P,A,A)}
    \E_3 (P,A,A) \ge 2^{-9} \gamma^{1/2} \sigma^5_P (A) \E(A) |A|^{-9} \,.
\end{equation}
\end{corollary}
\begin{proof}
Let
$\sigma = \sigma_P (A)$.
On the one hand, using Lemma \ref{l:sumsets_disjoint} with $A=A$, $B=-A$, we construct the family
of disjoint sets $P_j \subseteq P \cap (A-a_j)$, $a_j \in A$, $|P_j| \ge 2^{-3} \sigma |A|^{-1}$,
the number $s$ satisfies (\ref{f:s_bound+}).
Put $A_j = P_j + a_j \subseteq A$.
Thus, by connectedness of our set $A$, we have
$$
    \E_3 (P,A) \ge \sum_{j=1}^s \E_3 (P_j,A) = \sum_{j=1}^s \E_3 (A_j,A)
        \ge
            \gamma \left( \frac{|A_j|}{|A|} \right)^6 \E_3 (A)
                \ge
                2^{-18} \gamma \frac{\sigma^6}{|A|^{12}} \E_3 (A) \,.
$$
On the other hand, applying Lemma \ref{corpop}, we get
$$
    \sigma^4 \E^2 (A) |A|^{-6}
        \le
            \E_3 (A) \E_3 (P,A) \,.
$$
Combining the last two inequalities, we  obtain
bound (\ref{f:E_3(P,A,A)}).
This concludes the proof.
$\hfill\Box$
\end{proof}

\section{On Gowers norms}
\label{sec:Gowers}

    The notion of Gowers norms was introduced in papers \cite{Gow_4,Gow_m}.
    At the moment it is a very important tool of investigation in wide class
    of problems of additive combinatorics
    (see e.g. \cite{GT_great}---\cite{Gow_m}, \cite{Samorod_false}, \cite{Samorod})
    as well as in ergodic theory (see e.g. \cite{Austin}, \cite{Fu}, \cite{HK1}, \cite{HK2},
    \cite{Tao1}, \cite{Tao2}, \cite{TZ}, \cite{Z1}, \cite{Z2}).
    Recall the definitions.

  Let
  $G$
  be a finite
  set,
  and $N=|G|$.
Let also $d$ be a positive integer, and
$$\{ 0,1 \}^d = \{ \omega = (\omega_1,\dots, \omega_d) ~:~ \omega_j \in \{0,1\},\, j=1,2,\dots,d \}$$
be the ordinary $d$---dimensional cube.
For $\omega \in \{0,1\}^d$ denote by $|\omega|$
the sum
$\omega_1 + \dots + \omega_d$.
Let also $\mathcal{C}$ be the operator of complex conjugation.
Let $\v{x} = (x_1,\dots,x_d), \v{x}'=(x'_1,\dots,x'_d)$ be two arbitrary vectors from $G^d$.
By $\v{x}^\o = (\v{x}^\o_1, \dots, \v{x}^\o_d)$ denote the vector
\begin{displaymath}
    \v{x}^\o_i = \left\{ \begin{array}{ll}
                        x_i & \mbox{ if } \o_i = 0, \\
                        x'_i & \mbox{ if } \o_i = 1. \\
                        \end{array} \right.
\end{displaymath}
Thus $\v{x}^\o$ depends on $\v{x}$ and $\v{x}'$.


    Let $f : G^d \to \C$ be an arbitrary function.
    We will write $f(\v{x})$ for $f(x_1,\dots,x_d)$.

\begin{definition}
    {\it Gowers $U^d$--norm} (or $d$--uniformity norm) of the function $f$ is the following expression
\begin{equation}\label{f:G_norm_d'}
    \| f \|_{U^d}
            =
                \left( N^{-2d} \sum_{\v{x}\in G^d}\, \sum_{\v{x}' \in G^d}
                    \prod_{\o \in \{ 0,1 \}^d} \mathcal{C}^{|\o|} f (\v{x}^{\o}) \right)^{1/2^d} \,.
\end{equation}
\end{definition}

A sequence of $2^d$ points $\v{x}^\o \in G^d$, $\o \in \{ 0,1 \}^d$
is called {\it $d$--dimensional cube in $G^d$}
or just a $d$--dimensional cube.
Thus the summation in formula (\ref{f:G_norm_d'})
is taken over all cubes of $G^d$.
For example, $\{ (x,y), (x',y), (x,y'), (x',y') \}$, where $x,x',y,y'\in G$
is a two--dimensional cube in $G \m G$.
In the case Gowers norm is called rectangular norm.

For $d=1$ the expression above gives a semi--norm but for $d\ge 2$
Gowers norm is a norm.
In particular, the triangle inequality holds \cite{Gow_m}
\begin{equation}\label{e:triangle}
    \| f+g \|_{U^d} \le \| f \|_{U^d} + \| g \|_{U^d} \,.
\end{equation}
One can prove also (see \cite{Gow_m}) the following monotonicity relation.
Let $f_{x_d} (x_1, \dots, x_{d-1}) := f (x_1, \dots, x_{d})$. Then
\begin{equation}\label{e:Gowers_mon}
    N^{-1} \sum_{x_d \in G} \| f_{x_d} \|^{2^{d-1}}_{U^{d-1}} \le \| f \|^{2^{d-1}}_{U^d}
\end{equation}
for all $d\ge 2$.

If $\Gr = (\Gr,+)$ is a finite
Abelian group with additive group operation $+$, $N=|\Gr|$
then one can "project"\, the norm above onto the group
$\Gr$ and obtain the ordinary ("one--dimensional") Gowers norm.
In other words, we put the function $f(x_1,\dots,x_d)$ in formula
(\ref{f:G_norm_d'})
equals "one--dimensional"\, function  $f(x_1,\dots,x_d) := f({\rm pr} (x_1,\dots,x_d))$,
where ${\rm pr} (x_1,\dots,x_d) = x_1+\dots+x_d$.
Denoting the obtained norm as $U^d$,
we
have
an analog of (\ref{e:Gowers_mon}), see \cite{Gow_m}, \cite{TV}
\begin{equation}\label{e:Gowers_mon_1}
    \| f \|_{U^{d-1}} \le \| f \|_{U^d}
\end{equation}
for all $d\ge 2$.
It is convenient to write
\begin{equation}\label{f:G_norm_d'_UU-}
    \| f \|_{\U^d} =
    N^{-d+1}
                    \sum_{\v{x}\in \Gr^d}\, \sum_{\v{x}' \in \Gr^d}
                    \prod_{\o \in \{ 0,1 \}^d} \mathcal{C}^{|\o|} f ( {\rm pr} (\v{x}^{\o}))
    =
\end{equation}
\begin{equation}\label{f:G_norm_d'_UU}
    =
        \sum_x \sum_{h_1,\dots,h_{d}}\, \prod_{\o \in \{ 0,1 \}^{d}}
                \mathcal{C}^{|\o|} f(x + \o\cdot \v{h}) \,.
\end{equation}
In the case $f=A$, where $A\subseteq \Gr$ is a set, we have by formula (\ref{f:G_norm_d'_UU}) that
$$
    \| A \|_{\U^d} = \sum_{s_1,\dots,s_d} |A_{\pi(s_1,\dots,s_d)}| \,,
$$
where $\pi(s_1,\dots,s_d)$ is a vector with $2^d$ components, namely,
$$
    \pi(s_1,\dots,s_d) = \left( \sum_{j=1}^d s_j \eps_j \right)\,, \quad \quad \eps_j \in \{ 0,1 \}^d \,.
$$
Note also
\begin{equation}\label{f:Gowers_sq_A}
    \| A \|_{\U^{d+1}} = \sum_{s_1,\dots,s_d} |A_{\pi(s_1,\dots,s_d)}|^2 \,.
\end{equation}
Further, $\| A \|_{\U^1} = \E_1 (A) = |A|^2$ and $\| A \|_{\U^2} = \E (A)$.

\bigskip

In definitions (\ref{f:G_norm_d'}), (\ref{f:G_norm_d'_UU-})
we have used the size of the set G/group $\Gr$.
The results of the paper are local, in the sense that they  do not use the
cardinality of the
container group $\Gr$.
Thus it is natural to ask about the possibility to obtain an analog of (\ref{e:Gowers_mon_1}), say,
without any $N$ in the definition.
That is our simple result in the direction.

\begin{proposition}
    Let $A \subseteq \Gr$ be a set.
    Then for any integer $k\ge 2$ one has
\begin{equation}\label{f:Gowers_A}
    \| A \|_{\U^{k+1}} \ge \frac{\| A \|^{(3k-2)/(k-1)}_{\U^{k}}}{\| A \|^{2k/(k-1)}_{\U^{k-1}}} \,.
\end{equation}
In particular,
\begin{equation}\label{f:Gowers_A_0}
        \| A \|_{\U^3} \ge \frac{\E^4 (A)}{|A|^8} \,.
\end{equation}
\label{p:Gowers_A}
\end{proposition}
\begin{proof}
We have
$$
    \| A \|_{\U^k} = \sum_{s_1,\dots,s_k} |A_{\pi(s_1,\dots,s_k)}|
        =
            \sum_{s_1,\dots,s_{k-1}} \sum_{s_k} \sum_z A_{\pi(s_1,\dots,s_{k-1})} (z) A_{\pi(s_1,\dots,s_{k-1})} (z+s_k)
                =
$$
\begin{equation}\label{tmp:01.02.2014_1}
                =
                    \sum_{s_1,\dots,s_{k-1}} \sum_z A_{\pi(s_1,\dots,s_{k-1})} (z) \cdot |A_{\pi(s_1,\dots,s_{k-1})}| \,.
\end{equation}
Thus, if the summation in (\ref{tmp:01.02.2014_1}) is taken over the set
\begin{equation}\label{tmp:Q_def}
    Q_k := \{ (s_1,\dots,s_{k-1}) ~:~ |A_{\pi(s_1,\dots,s_{k-1})}| \ge \| A \|_{\U^k} (2k \| A \|_{\U^{k-1}} )^{-1} \}
\end{equation}
then it gives us $(1-1/2k)$ proportion of the norm $\| A \|_{\U^k}$.
Let us estimate the size of $Q_k$.
Clearly,
$$
    |Q_k| \| A \|_{\U^k} (2k \| A \|_{\U^{k-1}} )^{-1}
        \le
            \sum_{s_1,\dots,s_{k-1}} |A_{\pi(s_1,\dots,s_{k-1})}| = \| A \|_{\U^{k-1}}
$$
and whence $|Q_k| \le 2k \| A \|^2_{\U^{k-1}} \| A \|^{-1}_{\U^k}$.
Certainly, the same bound holds for the cardinality of any set of tuples $(s_{i_1},\dots,s_{i_{k-1}})$
defined similar to (\ref{tmp:Q_def}) having the size $k-1$.
Hence, by the projection results, see e.g. \cite {Bol_Th}, we see that the summation
in (\ref{tmp:01.02.2014_1}) is taken over a set $\mathcal{S}$ of vectors $(s_1,\dots,s_k)$ of size at most
$
     ( 2k \| A \|^2_{\U^{k-1}} \| A \|^{-1}_{\U^k} )^{k/(k-1)}
$.
Returning to (\ref{tmp:01.02.2014_1}) and using the Cauchy--Schwarz inequality
as well as formula (\ref{f:Gowers_sq_A}), we obtain
$$
    2^{-2} \| A \|^2_{\U^k}
        \le
           \left( \sum_{(s_1,\dots,s_k) \in \mathcal{S}} |A_{\pi(s_1,\dots,s_k)}| \right)^2
            \le
                |\mathcal{S}| \sum_{s_1,\dots,s_k} |A_{\pi(s_1,\dots,s_k)}|
                    \le
$$
$$
                    \le
                        ( 2k \| A \|^2_{\U^{k-1}} \| A \|^{-1}_{\U^k} )^{k/(k-1)} \| A \|_{\U^{k+1}} \,.
$$
The last inequality implies that
$$
    \| A \|_{\U^{k+1}} \ge C_k \frac{\| A \|^{(3k-2)/(k-1)}_{\U^{k}}}{\| A \|^{2k/(k-1)}_{\U^{k-1}}} \,,
$$
where $0<C_k<1$ depends on $k$ only.
Using the tensor trick we obtain the result.
This completes the proof.
$\hfill\Box$
\end{proof}

\begin{remark}
    Estimate (\ref{f:Gowers_A}) is sharp as an example of a sufficiently dense random subset of a group $\Gr$ shows.
    For higher Gowers norms one can show by induction a similar sharp inequality
    $\| A\|_{\U^k} \ge \E(A)^{2^k-k-1} |A|^{-(3\cdot2^k-4k-4)}$.
    It demonstrates expected exponential (in terms of $\E(A)$) growth  of the norms.
\end{remark}

In the next section we will need in a statement, which is generalizes lower bound for $U^3$--norm (\ref{f:Gowers_A_0}).

\begin{lemma}
     Let $A,B\subseteq \Gr$ be two sets.
     Then
\begin{equation}\label{f:U^3(A,B)}
    \sum_{s_1,s_2} \left( \sum_{x} A(x) B(x+s_1) A(x+s_2) B(x+s_1+s_2) \right)^2
        \ge
            \frac{\E^4 (A,B)}{|A|^4 |B|^4} \,.
\end{equation}
\end{lemma}
\begin{proof}
We use the same arguments as in the proof of Proposition \ref{p:Gowers_A}.
One has
$$
    \E (A,B) = \sum_{s_1,s_2} \sum_{x} A(x) B(x+s_1) A(x+s_2) B(x+s_1+s_2)
        =
$$
$$
        =
            \sum_{x} \sum_{s_1} B^A_{s_1} (x) |B^A_{s_1}| = \sum_x \sum_{s_2} A_{s_2} (x) |B_{s_2}| \,.
$$
Because of $\sum_s |A_s| = |A|^2$, $\sum_s |B_s| = |B|^2$, we get
for the set $\mathcal{S}$ above that\\ $|\mathcal{S}| \ll (|A|^2 |B|^2 \E^{-1} (A,B))^2$.
Thus
$$
    \E^2 (A,B)
        \ll
            |\mathcal{S}| \sum_{s_1,s_2} \left( \sum_{x} A(x) B(x+s_1) A(x+s_2) B(x+s_1+s_2) \right)^2
$$
and the result follows.
$\hfill\Box$
\end{proof}

\bigskip

Using (\ref{f:Gowers_A_0}) and the Cauchy--Schwarz inequality, we have a consequence.

\begin{corollary}
    Let $A\subseteq \Gr$ be a set and $|A - A|\le K|A|$ or $|A + A|\le K|A|$.
    Then
    $$
        \| A \|_{\U^3} \ge \frac{|A|^4}{K^4} \,.
    $$
\label{c:U^3&doubling}
\end{corollary}

Inequality (\ref{f:Gowers_A_0}) gives us a relation between $\| A\|_{\U^3} = \sum_s \E (A_s)$ and $\E(A)$.
W.T. Gowers (see \cite{Gow_m}) constructed a set $A$ having a random behavior in terms of $\E(A)$
(more precisely, he constructed a uniform set, that is having small Fourier coefficients, see \cite{Gow_m})
such that for all $s$ the sets $A_s$ have
non--random (non--uniform) behavior in terms of $\E(A_s)$.
Nevertheless, it is natural to ask about the possibility to find an $s\neq 0$ with a weaker notion of randomness,
that is $\E (A_s) \ll |A_s|^{3-c}$, $c>0$.
This question was asked to the author by T. Schoen.
We give an affirmative answer on it.

\begin{theorem}
    Let $A\subseteq \Gr$ be a set, $\E(A) = |A|^3 /K$.
Suppose that for all $s\neq 0$ the
following holds
\begin{equation}\label{cond:A_s_bounded}
    |A_s| \le \frac{M|A|}{K} \,,
\end{equation}
 where $M\ge 1$ is a real number.
 Let $K^4 \le M |A|$.
Then there is $s\neq 0$ such that $|A_s| \ge |A|/2K$ and
\begin{equation}\label{f:E(A_s)}
    \E (A_s) \ll \frac{M^{93/79}}{K^{1/198}} \cdot |A_s|^3 \,.
\end{equation}
\label{t:E(A_s)}
\end{theorem}
\begin{proof}
Let
$$
    P 
        := \{ s ~:~ |A_s| \ge |A|/2K \} \,.
$$
Find the number $L$ satisfying $L := \max_{s\in P} |A_s|^3 \E^{-1} (A_s)$.
In other words for all $s \in P$, one has $\E(A_s) \ge |A_s|^3 / L$.
Our task is to find a lower bound for $L$.

Put
$$
    \Cf(x,y) = \Cf_3 (A) (x,y) := |A \cap (A-x) \cap (A-y)|
$$
and
$$
    \t{\Cf} (x,y) = \Cf_4 (A) (x,y,x-y) := |A \cap (A-x) \cap (A-y) \cap (A-x + y)|
    \,.
$$
Clearly,
\begin{equation}\label{f:C_A_s}
    \t{\Cf} (x,y) \le \Cf(x,y) \le \min\{ |A_x|, |A_y| \} \,.
\end{equation}
for any $x,y$.
Put also
$$
    \mathcal{P} := \{ (x,y) ~:~ \t{\Cf} (x,y) \ge |A|/(4KL) \} \,.
$$
We will write $\mathcal{P}_x := \mathcal{P} \cap (\{x\} \m \Gr)$,
and $\mathcal{P}_y := \mathcal{P} \cap (\Gr \m \{y\})$.
Put also
$$
    \mathcal{P}^\la := \mathcal{P} \cap \{(x,y) ~:~ x - y = \lambda \} \,.
$$

Our first lemma says that the size of $\mathcal{P}$ and some characteristics of the set
can be
estimated
in terms of $L$ and $M$.

\begin{lemma}
We have
\begin{equation}\label{f:P_size1}
    \frac{|A|^2}{4LM^2} \le |\mathcal{P}| \le 4L|A|^2 \,.
\end{equation}
Further, for any nonzero $y$ and $\la$ the following holds
\begin{equation}\label{f:P_size2}
    |\mathcal{P}_y|\,, |\mathcal{P}^\la| \le \frac{4M^2 L|A|}{K} \,.
\end{equation}
\label{l:P_size}
\end{lemma}
\begin{proof}
By the Cauchy--Schwarz inequality, we have
$$
    \sum_{s\in P} |A_s|^3 \ge 2^{-1} \E_3 (A) \ge 2^{-1} |A|^4 / K^2 \,.
$$
Hence
$$
    \frac{|A|^4}{2K^2 L}
        \le
    \frac{1}{L} \sum_{s\in P} |A_s|^3
        \le
            \sum_{s\in P} \E(A_s) \le \sum_s \E (A_s)
            =
$$
$$
            =
            \sum_s \sum_{z,x,y} A_s (z) A_s (z+x) A_s (z+y) A_s (z+x-y) = \sum_{x,y} \t{\Cf}^2 (x,y) \,.
$$
We can assume that $|A| \ge 4 K L^{1/2}$ because otherwise  the result is trivial
in view of the condition $K^4 \le M |A|$.
Since
$$
    \sum_{x,y} \t{\Cf} (x,y) = \E (A) = \frac{|A|^3}{K} \,,
$$
it follows
by (\ref{f:C_A_s}) and the assumption  $|A| \ge 4 K L^{1/2}$ that
$$
    \frac{|A|^4}{4K^2 L} \le \sum_{0\neq (x,y) \in \mathcal{P}} \t{\Cf}^2 (x,y)
        \le
            \left( \frac{M|A|}{K} \right)^2 |\mathcal{P}| \,.
$$
In other words $\frac{|A|^2}{4LM^2} \le |\mathcal{P}|$.
On the other hand
$$
    \frac{|A|}{4KL} |\mathcal{P}| \le \sum_{(x,y) \in \mathcal{P}} \t{\Cf} (x,y) \le \sum_{x,y} \t{\Cf} (x,y)
        =
            \E(A) = \frac{|A|^3}{K}
$$
and we obtain the required upper bound for the size of $\mathcal{P}$.

Further, for any fixed $y \neq 0$ the following holds
$$
    |\mathcal{P}_y| \le \frac{4KL}{|A|} \sum_x \t{\Cf} (x,y) = \frac{4KL}{|A|} |A_{- y}| |A_y| \le \frac{4M^2 L|A|}{K} \,.
$$
Finally,
$$
    |\mathcal{P}^\la| = \sum_{(x,y) ~:~ x - y = \lambda} \mathcal{P} (x,y)
        \le
            \frac{4KL}{|A|} \sum_{(x,y) ~:~  x - y = \lambda}\, \sum_x \t{\Cf} (x,y)
                =
                    \frac{4KL}{|A|} |A_\la|^2
                    \le \frac{4M^2 L|A|}{K}
$$
as required.
$\hfill\Box$
\end{proof}

\bigskip

Now, we show that some norm of $\mathcal{P}$ is huge.
Actually, we use the function $\Cf$ not $\t{\Cf}$ in the proof.

\begin{lemma}
    One has
\begin{equation}\label{f:P_norm}
    \frac{|A|^{5}}{2^9 K L^4 M^6}
        \le
            \sum_{y,x',y'} |\sum_x \mathcal{P} (x,y+x) \mathcal{P} (x'+x,y'+x) |^2 \,.
\end{equation}
\label{l:P_norm}
\end{lemma}
\begin{proof}
As in the proof of Lemma \ref{l:P_size}, we get
$$
    \frac{|A|^3}{4KLM} \le \sum_{(x,y) \in \mathcal{P}} \Cf (x,y)
        =
            \sum_z A(z) \sum_{x,y} \mathcal{P} (x,y) A(z+x) A(z+y) \,.
$$
Using the Cauchy--Schwarz inequality, we obtain
$$
    \frac{|A|^5}{16 K^2 L^2 M^2}
        \le
            \sum_z A(z) \sum_{x,y}
                \sum_{x',y'} \mathcal{P} (x,y) A(z+x) A(z+y) \mathcal{P} (x',y') A(z+x') A(z+y')
    \le
$$
$$
    \le
        \sum_{x,y,x',y'} \mathcal{P} (x,y+x) \mathcal{P} (x'+x,y'+x) \Cf_4  (A) (y,x',y')
            \,.
$$
Applying the Cauchy--Schwarz inequality again, we have
$$
    \frac{|A|^{10}}{2^8 K^4 L^4 M^4}
        \le
            \sum_{y,x',y'} |\sum_x \mathcal{P} (x,y+x) \mathcal{P} (x'+x,y'+x) |^2
                \m
                    \E_4 (A)
                        \le
$$
$$
                        \le
                \sum_{y,x',y'} |\sum_x \mathcal{P} (x,y+x) \mathcal{P} (x'+x,y'+x) |^2
        \m
            \frac{2M^2 |A|^5}{K^3}
$$
because of $K^3\le K^4 \le M |A| \le M^2 |A|$ and hence
\begin{equation}\label{f:E_4_tmp'}
    \E_4 (A)
        \le
            |A|^4 + \frac{M^2 |A|^2}{K^2} \E (A)
                \le
                    \frac{2M^2 |A|^5}{K^3} \,.
\end{equation}
This concludes the proof of the lemma.
$\hfill\Box$
\end{proof}

\bigskip

In terms of the sets $\mathcal{P}^\la$ we can rewrite expression (\ref{f:P_norm}) as
$$
    \sum_{y,x',y'} |\sum_x \mathcal{P} (x,y+x) \mathcal{P} (x'+x,y'+x) |^2
        =
            \sum_{y,x',y'} | \sum_x \mathcal{P}^{-y} (x) \mathcal{P}^{x'-y'} (x'+x) |^2
                =
$$
$$
                =
                    \sum_{y,x',y'} | \sum_x \mathcal{P}^{y} (x) \mathcal{P}^{y'} (x'+x) |^2
                        =
                            \sum_{\la,\mu} \E (\mathcal{P}^{\la}, \mathcal{P}^{\mu}) \,.
$$
In view of estimate (\ref{f:P_norm}), a trivial inequality
$$
    \sum_{\la,\mu} \E (\mathcal{P}^{\la}, \mathcal{P}^{\mu})
        \le
            \sum_{\la,\mu} |\mathcal{P}^{\la}|^2 |\mathcal{P}^{\mu}|
                \ll
                    |\mathcal{P}|^2 \cdot \max_{\la \neq 0} |\mathcal{P}^{\la}|
$$
and bounds for size of $\mathcal{P}^{\la}$ of Lemma \ref{l:P_size}
the sets $\mathcal{P}^\la$ should look, firstly, like some sets with small doubling
(more precisely as sets with  large additive energy) and,
secondly, the large proportion of such sets must correlate to each other.
A model example here is $\mathcal{P}^\la (x) = Q(x+\a(\la))$, where $\a$ is an arbitrary function
and $Q$ is an arithmetic progression of size approximately $\Theta (|A|/K)$.
Now we prove that the example is the only case, in some sense.

\bigskip

Put $l = \log (LM)$.
Note that $l\le \log (KM)$ because if $L\ge K$ then the result is trivial.
We can suppose that
\begin{equation}\label{cond:L_tmp}
    2^{17} M^6 K^3 L^7 \le |A|
\end{equation}
because otherwise the result follows
immediately
in view of the condition $M |A| \ge K^4$.
Reducing zero terms, we have
$$
    \frac{|A|^{5}}{2^{10} K L^4 M^6}
        \le
    \sum_{\la,\mu \neq 0} \E (\mathcal{P}^{\la}, \mathcal{P}^{\mu})
$$
because of
Lemma \ref{l:P_size}, estimate (\ref{cond:L_tmp}), a bound
$$
    |\mathcal{P}^{0}| \le 4KL |A|^{-1} \sum_{x} |A_x| = 4KL |A|
$$
and a calculation
$$
    2(4KL|A|)^2 4L |A|^2 \le \frac{|A|^{5}}{2^{10} K L^4 M^6} \,.
$$
Below we will assume that any summation is taken over
nonzero indices
$\la$, $\mu$.
By Lemma \ref{l:P_size} and
a trivial estimate
$$
    \sum_{\la} |\mathcal{P}^{\la}| = |\mathcal{P}| \le 4 L |A|^2 \,,
        \quad \quad
    \E (\mathcal{P}^{\la}, \mathcal{P}^{\mu})
        \le
            |\mathcal{P}^{\la}| |\mathcal{P}^{\mu}|^2
$$
the following holds
$$
    \frac{|A|^{5}}{2^{11} K L^4 M^6}
            \le
    \sum_{\la,\mu ~:~ |\mathcal{P}^{\la}|,\, |\mathcal{P}^{\mu}| \ge \D_*} \E (\mathcal{P}^{\la}, \mathcal{P}^{\mu}) \,,
$$
where $\D_* = \frac{|A|}{2^{16} K L^6 M^6}$.
Using the pigeonhole principle and Lemma \ref{l:P_size}, we find a number $\D$
such that $\D_* \le \D \le \frac{4M^2 L|A|}{K}$ and
\begin{equation}\label{tmp:10.06.2013_1}
    \frac{|A|^{5}}{lK L^4 M^6}
        \ll
            \sum_{\mu ~:~ \D < |\mathcal{P}^{\mu}| \le 2\D} \,\,
                \sum_{\la :~ |\mathcal{P}^{\la}| \ge \D_*}
                    \E (\mathcal{P}^{\la}, \mathcal{P}^{\mu}) \,.
\end{equation}
From  (\ref{tmp:10.06.2013_1}) and the Cauchy--Schwarz inequality one can see
that the summation in the
formula is taken over
$$
    \E(\mathcal{P}^{\mu})
        \gg
            \frac{|\mathcal{P}^{\mu}|^3}{l^2 L^{14} M^{16}}
                := \eps |\mathcal{P}^{\mu}|^3 \,.
$$
By
inequality
(\ref{f:P_size1}) of Lemma \ref{l:P_size} there is $\mu$ with
\begin{equation}\label{tmp:11.03.2014_1}
    \zeta \frac{|A|^{3} \D}{K}
        :=
    \frac{|A|^{3} \D}{lK L^5 M^6}
        \ll
            \sum_{\la ~:~ |\mathcal{P}^{\la}| \ge \D_*}
                    \E (\mathcal{P}^{\la}, \mathcal{P}^{\mu}) \,.
\end{equation}
Put $Q=\mathcal{P}^{\mu}$.
We have $\E(Q) \ge \eps |Q|^3$.
Applying a trivial general
bound
$$
    \E(A,B) \le |A| |B| \m \max_{x} |A \cap (B-x)| \,,
$$
we get by Lemma \ref{l:P_size}
$$
    \frac{|A|^{3} \D}{lK L^5 M^6}
        \ll
            \D \frac{4M^2 L|A|}{K} \m \sum_{\la ~:~ |\mathcal{P}^{\la}| \ge \D_*} \max_{x} |\mathcal{P}^{\la} \cap (Q-x)| \,.
$$
Given
an arbitrary $\la$ let the maximum in the last formula is attained at point $x:=\a(\la)$.
Thus, we have
$$
    \frac{|A|^{2}}{l L^6 M^8}
        \ll
            \sum_{\la ~:~ |\mathcal{P}^{\la}| \ge \D_*} |\mathcal{P}^\la \cap (Q-\a(\la))|
                =
                    \sum_{\la ~:~ |\mathcal{P}^{\la}| \ge \D_*}
                        \sum_{x} \mathcal{P} (x,x-\la) Q(x+\a(\la)) \,.
$$
Hence
we find a set $Q$ of the required form, that is having the large additive energy and which is
correlates with sets $\mathcal{P}^{\la}$.
Now, we transform the obtained information into some knowledge about the original set $A$.

\bigskip

Using the definition of the set $\mathcal{P}$, we obtain
$$
    \frac{|A|^{3}}{l L^7 M^8 K}
        \ll
            \sum_{x,\la} Q(x+\a(\la)) \sum_z A (z) A (z+x) A (z+x-\la) A (z+\la)
                =
$$
\begin{equation}\label{tmp:10.06.2013_2}
                =
                    \sum_{z,\la} (Q\circ A_{-\la}) (z-\a(\la)) A_\la (z) \,.
\end{equation}
We know that $\E(Q) \ge \eps |Q|^3$.
By Balog--Szemer\'{e}di--Gowers Theorem \ref{BSG} we find $Q'\subseteq Q$, $|Q'| \gg \eps |Q|$
such that  $|Q'-Q'| \ll {\eps}^{-4} |Q'|$.
We
will prove shortly
that the set $Q$ in (\ref{tmp:10.06.2013_2}) can be replaced by a set $\t{Q}$, namely
\begin{equation}\label{tmp:10.06.2013_3}
    c(\eps) \frac{|A|^{3}}{K}
        \ll
            \sum_{z,\la} (\t{Q} \circ A_{-\la}) (z-\a(\la)) A_\la (z) \,,
\end{equation}
where $c(\eps)>0$ is some constant depends on $L$ and $M$ only
and $\t{Q}$ has small doubling.
Indeed, starting with (\ref{tmp:11.03.2014_1}),
put $Q'_1 = Q'$, and define inductively disjoint sets $Q'_j \subseteq Q$,
$B_j = \bigsqcup_{i=1}^j Q'_j$, $\bar{B}_j = Q\setminus B_j$, $j\in [s]$, applying Theorem \ref{BSG} to $\bar{B}_j$.
Put also  $\E(\bar{B}_j) = \nu_j |Q|^3 \le 8 \nu_j \D^3$.
If at some stage $j$
\begin{equation}\label{tmp:11.03.2014_2}
    \zeta \frac{|A|^{3} \D}{2K}
        >
            \sum_{\la ~:~ |\mathcal{P}^{\la}| \ge \D_*}
                    \E (\mathcal{P}^{\la}, \bar{B}_j)
\end{equation}
then we stop the algorithm.
Let the procedure works exactly $s$ steps
and put $B= B_s$, $\bar{B} = Q\setminus B$.
We claim that $s\ll \eps^{-1}$.
To prove this note that if (\ref{tmp:11.03.2014_2}) does not hold then
$$
    \zeta \frac{|A|^{3} \D}{K} \ll \E^{1/2} (\bar{B}_j) \sum_{\la ~:~ |\mathcal{P}^{\la}| \ge \D_*} |\mathcal{P}^{\la}|^{3/2}
        \ll
            \nu^{1/2}_j \D \frac{M^2L |A|}{K} L |A|^2 \,.
$$
In other words,
$\E (\bar{B}_j) \gg \eps |Q|^3$ and, hence,
$|Q'_j| \gg \eps |Q|$,   $|Q'_j-Q'_j| \ll {\eps}^{-4} |Q|$.
It means, in particular, that after $s\ll \eps^{-1}$ number of steps our algorithm
stops indeed.
At the last step, we get by the
construction
that
$$
    \zeta \frac{|A|^{3} \D}{2K}
        \le
            \sum_{\la ~:~ |\mathcal{P}^{\la}| \ge \D_*} \E (\mathcal{P}^{\la},Q)
                - \sum_{\la ~:~ |\mathcal{P}^{\la}| \ge \D_*} \E (\mathcal{P}^{\la}, \bar{B}, \bar{B})
        \le
$$
$$
        \le
            \sum_{\la ~:~ |\mathcal{P}^{\la}| \ge \D_*} \E (\mathcal{P}^{\la}, B, B)
                    +
                2\sum_{\la ~:~ |\mathcal{P}^{\la}| \ge \D_*}
                        2\E (\mathcal{P}^{\la}, B, \bar{B}) ) \,.
$$
Let us prove
the following estimate
$$
    \zeta \frac{|A|^{3} \D}{K}
        \ll
            \sum_{\la ~:~ |\mathcal{P}^{\la}| \ge \D_*}
                    \E (\mathcal{P}^{\la}, B) \,.
$$
If not then by the Cauchy--Schwarz inequality and the choice of $\bar{B}$,
we obtain
$$
    \left( \zeta \frac{|A|^{3} \D}{K} \right)^2
        \ll
            \left( \sum_{\la ~:~ |\mathcal{P}^{\la}| \ge \D_*} \E (\mathcal{P}^{\la}, B, \bar{B}) \right)^2
                \le
                    \sum_{\la ~:~ |\mathcal{P}^{\la}| \ge \D_*}
                    \E (\mathcal{P}^{\la}, B) \cdot  \zeta \frac{|A|^{3} \D}{2K}
$$
and we 
get a contradiction.
Hence
the following holds
$$
    \zeta \frac{|A|^{3} \D}{K}
        \ll
            \sum_{\la ~:~ |\mathcal{P}^{\la}| \ge \D_*}
                    \E (\mathcal{P}^{\la}, B) \,.
$$
Applying
the H\"{o}lder inequality, we find a set $Q'_j$ such that
$$
    \zeta \frac{|A|^{3} \D}{s^2 K}
        \ll
            \sum_{\la ~:~ |\mathcal{P}^{\la}| \ge \D_*}
                    \E (\mathcal{P}^{\la}, Q'_j) \,.
$$
So, putting $Q':=Q'_j$ we get (\ref{tmp:10.06.2013_3}) with $c(\eps) \gg \frac{\zeta \eps^2}{lL^2 M^2}$.
Of course, the summation in the
obtained
formula can be taken just over $\la$ with $|A_\la| \gg c(\eps) \frac{|A|}{K}$
and we will assume this.

\bigskip

Denote by $\Omega$ the set
$$
    \Omega := \{ (z,\la) ~:~ A_\la (z) =1 \,, \mbox{ and } (\t{Q} \circ A_{-\la}) (z-\a(\la)) \ge \frac{c(\eps)|A|}{2K} \} \,.
$$
From (\ref{tmp:10.06.2013_3}) and our assumption (\ref{cond:A_s_bounded})
we have $|\Omega| \gg c(\eps) M^{-1} |A|^2$.
On the other hand, considering $\Omega_\la := \{ z~:~ (z,\la) \in \Omega \}$ for any fixed $\la$, one has
$$
    |\Omega_{\la}| \frac{|A|}{K} c(\eps) \ll \sum_z A_\la (z) (\t{Q} \circ A_{-\la}) (z-\a(\la))
        \le
            |A_\la| |Q|
                \ll
                    |A_\la| \D
                \ll \frac{M^3 L |A|^2}{K^2} \,.
$$
Hence there are at least $\gg c^2 (\eps) K M^{-4} L^{-1} |A|$ sets $A_\la$
such that there exists some $z=z(\la)$ with $(\t{Q} \circ A_{\la}) (z-\a(\la)) \gg c(\eps) |A|/2K$.
Denote the set of these $\la$ by $T$.
For any such $A_\la$ there exists a shift of the set $\t{Q}$ such that
$|A_\la \cap (\t{Q} + w(\la))| \gg c(\eps) |A|/2K \gg_{\eps,M} |A_\la|$.
Put $A'_\la := A_\la \cap (\t{Q} + w(\la))$.
We have by Lemma \ref{l:Plunnecke} that for any $\la_1,\la_2 \in T$ the following holds
$$
    |A'_{\la_1}+A'_{\la_2}| \le |\t{Q}+\t{Q}| \ll \eps^{-8} \frac{M^2 L|A|}{K} \,.
$$
In particular,
$$
    \E(A'_{\la_1}, A'_{\la_2}) \gg \frac{c^4(\eps) \eps^8 |A|^3}{M^2 L K^3} \,.
$$
Finally, using (\ref{f:E_4_tmp'}) as well as Lemma \ref{l:E_3_A_s} with $k=l=2$, we obtain
$$
    \frac{|A|^5}{l^{64} M^{458} L^{395} K}
        \ll
    \frac{\zeta^8 \eps^{24} |A|^5}{l^8 M^{26} L^{19} K}
        \ll
    \frac{c^8(\eps) \eps^8 |A|^5}{M^{10} L^3 K}
        \ll
    |T|^2 \cdot \frac{c^4(\eps) \eps^8 |A|^3}{M^2 L K^3}
        \ll
            \sum_{\la_1, \la_2 \in T} \E(A'_{\la_1}, A'_{\la_2})
            \le
$$
$$
            \le
    \sum_{\la_1, \la_2 \in T} \E(A_{\la_1}, A_{\la_2}) \le \E_4 (A) \le \frac{2M^2 |A|^5}{K^3}
$$
with
the required lower bound for $L$.
This completes the proof of Theorem \ref{t:E(A_s)}.
$\hfill\Box$
\end{proof}

\bigskip

We finish the section by analog of Definition \ref{def:conn},
which we will use in the next section.

\begin{definition}
For $\beta,\gamma \in [0,1]$ a  set $A$ is called  $U^k (\beta,\gamma)$--connected
    if for any $B \subseteq A$, $|B| \ge \beta|A|$
    the following holds
    $$
        \| B \|_{\U^k} \ge \gamma \left( \frac{|B|}{|A|} \right)^{2^k} \| A \|_{\U^k} \,.
    $$
\label{def:Uk-conn}
\end{definition}

Again, if, say, $\gamma^{-1} |A|^8 / |A\pm A|^4 \ge \|A \|_{\U^3}$ then by inequality (\ref{f:Gowers_A_0})
one can  see that $A$ is $U^3 (\beta,\gamma)$--connected for any $\beta$.
The existence of $U^k (\beta,\gamma)$--connected subsets in an arbitrary set
is discussed in the Appendix.

\section{Self--dual sets}
\label{sec:structural2}

Inequality (\ref{f:Gowers_A_0}) gives us a relation between $\| A\|_{\U^3}$ and $\E(A)$.
It attaints at a random subset $A$ of $\Gr$,
where by randomness we mean that each element of $A$ belongs to the set with probability $\E(A)/|A|^3$.
On the other hand,
it is easy to see that an upper bound takes place
\begin{equation}\label{f:U_3_E_3}
    \| A \|_{\U^3} = \sum_s \E(A_s) \le \sum_s |A_s|^3 = \E_3 (A) \,.
\end{equation}
A weaker estimate follows from (\ref{f:U_3_E_3}) combining with  the Cauchy--Schwarz inequality
\begin{equation}\label{f:U_3_E_3'}
    \| A \|^2_{\U^3} \le \E_4 (A) \E(A) \,.
\end{equation}
In the section we consider sets having critical relations between $\| A \|_{\U^3}$ and $\E_4 (A)$, $\E(A)$
that is the sets satisfying the reverse inequality to (\ref{f:U_3_E_3'})
(actually, we use a slightly stronger estimate then reverse to (\ref{f:U_3_E_3})).
It turns out that they are exactly which we called self--dual sets.

\bigskip


Let us recall a result on large deviations.
The following variant
can be found in \cite{GreenA+A}.

\begin{lemma}
    Let $X_1,\dots,X_n$ be independent random variables with $\mathbb{E} X_j = 0$
    and $\mathbb{E} |X_j|^2 = \sigma_j^2$.
    Let $\sigma^2 = \sigma_1^2 + \dots + \sigma_n^2$.
    Suppose that for all  $j\in [n]$, we have $|X_j| \le 1$.
    Let also $a$ be a real number such that  $\sigma^2 \ge 6 na$.
    Then
    $$
        \mathbb{P} \left( \left| \frac{X_1+\dots + X_n}{n} \right| \ge a \right) \le 4 e^{-n^2 a^2 / 8\sigma^2} \,.
    $$
\label{l:large_deviations}
\end{lemma}

We need in a combinatorial lemma.

\begin{lemma}
    Let $\D$, $\sigma$, $C>1$ are positive numbers, $t$ be a positive integer,
    and $M_1, \dots, M_t$ be sets, $\D \le |M_j| \le C\D$, $j\in [t]$,
    $\sigma \le 10^{-4} t^2 \D$,
    where
$$
    \sigma:= \sum_{i,j=1}^t |M_i \bigcap M_j| \,.
$$
    Then
    there are at least $\frac{t^2 \D}{16(2C+1) \sigma}$ disjoint sets
    $\t{M}_l \subseteq M_{i_l}$
    such that
    $|\t{M}_l| \ge \frac{\D}{4(2C+1)}$.
\label{l:disjoint_probab}
\end{lemma}
\begin{proof}
    We will choose our sets $\t{M}_i$ randomly with probability at least $1/4$.
    First of all, we note that
\begin{equation}\label{tmp:25.03.2014_1}
    10^{-4} t^2 \D \ge \sigma \ge \sum_{i=1}^t |M_i| \ge t \D \,.
\end{equation}
    Put $p=t\D 2^{-1} \sigma^{-1}$.
    In view of (\ref{tmp:25.03.2014_1}), we get $p\in (0,1/2]$.
    Let us form a new family of sets taking a set $M_i$ from $M_1,\dots,M_t$ uniformly and independently
    with probability $p$.
    Denote the obtained family as $M'_1,\dots, M'_s$.
    By Lemma \ref{l:large_deviations} and bound (\ref{tmp:25.03.2014_1}), we have after some calculations that
\begin{equation}\label{tmp:25.03.2014_2}
    2^{-1} pt \le s \le 2pt
\end{equation}
    with probability at least $3/4$.
    Further the expectation of $\sigma$ equals
$$
    \mathbb{E} \sum_{i,j=1}^t |M_i \bigcap M_j|
        =
            \sum_x \sum_{i=1}^t \mathbb{E} M_i (x) + \sum_x \sum_{i,j=1,\,i\neq j}^t \mathbb{E} M_i (x) M_j (x)
                =
$$
$$
                =
                    p \sum_{i=1}^t |M_i| + p^2 \sum_{i,j=1,\,i\neq j}^t |M_i \bigcap M_j|
                        \le
                            Cpt\D + p^2 \sigma
                                \le
                                    (2C+1)p^2 \sigma
$$
by our choice of $p$.
Hence, by Markov inequality, with probability at least $1/2$ one has
$$
    \sum_{i,j=1}^s |M'_i \bigcap M'_j| \le (4C+2)p^2 \sigma
$$
and by the Cauchy--Schwarz inequality, we get
$$
    |\bigcup_{i=1}^s M'_i| \ge \frac{(\sum_{i=1}^s |M'_i| )^2}{\sum_{i,j=1}^s |M'_i \cap M'_j|}
        \ge
            \frac{s^2 \D^2}{(4C+2)p^2 \sigma} \ge \frac{2s^2 \sigma}{(2C+1)t^2} := q \,.
$$
Now we take disjoint subsets $M''_i \subseteq M'_i$, $i\in [s]$.
Thus
\begin{equation}\label{tmp:25.03.2014_3}
    2 \sum_{i ~:~ |M''_i| \ge q (2s)^{-1}} |M''_i| \ge \sum_{i} |M''_i| \ge q \,.
\end{equation}
Finally, we put $\t{M}_i = M''_i$.
By our choice of parameters and estimates (\ref{tmp:25.03.2014_2}) the following holds
$$
    |\t{M}_i| \ge \frac{q}{2s} = \frac{s\sigma}{(2C+1)t^2} \ge \frac{p\sigma}{(4C+2) t} = \frac{\D}{(8C+4)} \,.
$$
Similarly, the number $n$ of the sets $\t{M}_i$ can be estimated from (\ref{tmp:25.03.2014_3})
$$
    n \ge \frac{q}{2\D} = \frac{s^2 \sigma}{(2C+1)t^2 \D} \ge \frac{p^2 \sigma}{(8C+4) \D} = \frac{t^2 \D}{(32C+16) \sigma} \,.
$$
This completes the proof.
$\hfill\Box$
\end{proof}

\bigskip

Let us remark an interesting consequence of Lemma \ref{l:disjoint_probab}.

\bigskip

\begin{corollary}
    Let $A\subseteq \Gr$ be a
    $(2,\beta,\gamma)$--connected set, $\beta\le 0.5$ be a constant.
    Then there is a set
    $P \subseteq \{ x~:~ (A\c A) (x) \ge \D \}$
    satisfies
    $\E^P (A) \gg \E (A) \log^{-1} |A|$,
    and there are
    $k \gg  \gamma |A| \D^{-1} \log^{-1} |A|$ disjoint sets
    $\t{A}_j \subseteq A_{s_j}$ with  $|\t{A}_j| \gg \D$.
\label{c:disjoint_P}
\end{corollary}
\begin{proof}
    Using Lemma \ref{l:eigen_A'}, we find $A'$, $|A'| \ge |A|/2$ such that estimate (\ref{f:eigen_A'}) takes place.
    We want to apply Lemma \ref{l:disjoint_probab} to the sets $A'_s \subseteq A_s$, $s\in P'$,
    where $P'=\{ x~:~ (A'\c A') (x) \sim \Delta \}$, $\E^{P'} (A') \gg \E (A') \log^{-1} |A|$.
    Of course, such set $P'$ exists by the pigeonhole principle.
    To apply Lemma \ref{l:disjoint_probab}, we need to calculate the quantity $\sigma$
    $$
        \sigma := \sum_{s,t\in P'} |A'_s \bigcap A'_t| = \sum_{x,y} \Cf_3 (A') (x,y) P'(x) P'(y) \,.
    $$
    By the last identity and estimate (\ref{f:eigen_A'}) (for details, see \cite{s_mixed}), we get
    $$
        \sigma \ll \frac{\E(A)}{|A|} \cdot |P'| \,.
    $$
    Applying Lemma \ref{l:disjoint_probab}, we find disjoint sets $\t{A}_j \subseteq A'_{s_j} \subseteq A_{s_j}$,
    $s_j \in P'$, $j\in [k]$ such that
    $$
        k \gg \frac{|P'|^2 \D |A|}{\E(A) |P'|} \gg \gamma \frac{|A|}{\D \log |A|} \,.
    $$
    In the last inequality we have used   $(2,\beta,\gamma)$--connectedness  of $A$.
    To complete the proof note that $P' \subseteq \{ x~:~ (A\c A) (x) \ge \D \}$.
$\hfill\Box$
\end{proof}

\bigskip

Clearly, the bound on $k$ in Corollary \ref{c:disjoint_P} is the best possible up to logarithms.
Calculating $\E(A,A_j)/|A_j|$ and comparing its with $\E_3$ (see \cite{s_mixed}) one can obtain an alternative
proof of lower bounds for $|A\pm A_s|$ as of section \ref{sec:sumsets}.
Another result on a family with disjoint $A_s$ is proved in Proposition \ref{p:disjoint_via_A-A_s} below.

\bigskip

Now we are able to
obtain
the main result of the section.

\begin{theorem}
    Let $A\subseteq \Gr$ be a set, and $M\ge 1$ be a real number.
    Put $l=\log |A|$.
    Suppose that $A$ is
    $U^3(\beta,\gamma)$ and
    $(2,\beta,\gamma)$--connected with $\beta \le 0.5$.
    Then inequality
\begin{equation}\label{cond:E_3_and_E_critical_cor}
    \| A \|^2_{\U^3} \gg_{M} \E_4 (A) \E(A)
\end{equation}
    takes place iff there is  a positive real $\D \sim_{M,\,l} \E_3(A) \E(A)^{-1}$ and a set
    $$
        P \subseteq \{ s\in A-A ~:~ \D < |A_s| \} \,,
    $$
    such that $|P| \gg_{M,\,l} |A|$, $P=-P$, further,
\begin{equation}\label{f:self-dual0}
    \E^P (A) \gg_{M,\,l} \E(A) \,,\quad \quad \E^P_3 (A) \gg_{M,\,l} \E_3 (A)
        \,,\quad \quad \E^P_4 (A) \gg_{M,\,l} \E_4 (A) \,.
\end{equation}
    and
    such that for any $s\in P$ there is $H^s \subseteq A_s$,
    $|H^s| \gg_{M,\,l} \D$,
    with
\begin{equation}\label{f:self-dual1}
    |H^s - H^s| \ll_{M,\,l} |H^s| \,,
\end{equation}
    and $\E(A,H^s) \ll_{M,\,l} |H^s|^3$.\\
    Moreover there are disjoint sets $H_j \subseteq A_{s_j}$, $|H_j| \gg_{M,\,l} \D$,
    $s_j \in P$,  $j\in [k]$
    such that all $H_j$ have small doubling property (\ref{f:self-dual1}),
    $\E(A,H_j) \ll_{M,\,l} |H_j|^3$ and  $k\gg_{M,\,l} |A| \D^{-1}$.
\label{t:self-dual}
\end{theorem}
\begin{proof}
    Put $a=|A|$, $\E = \E(A)$, $\E_3 = \E_3 (A)$, $\E_4 = \E_4 (A)$.
    Let us begin with the necessary condition.
    Using Lemma \ref{l:eigen_A'}, we find $A'$, $|A'| \ge |A|/2$ such that estimate (\ref{f:eigen_A'}) takes place.
    Because of $A$ is
    $U^3(\beta,\gamma)$ and
    $(2,\beta,\gamma)$--connected with $\beta \le 0.5$,
    we have $\| A' \|_{\U^3} \sim \| A \|_{\U^3}$ and $\E(A') \sim \E(A)$.
    Combining assumption (\ref{cond:E_3_and_E_critical_cor})
    with the Cauchy--Schwarz inequality, we get
\begin{equation}\label{f:E_4_begin}
    \| A \|^2_{\U^3} \gg_{M} \E_4 (A) \E(A) \ge \E^2_3 (A)
\end{equation}
    In particular, by the last inequality and (\ref{f:U_3_E_3}), (\ref{f:U_3_E_3'}), we obtain
$$
    \E^2_3 (A') \ge \| A' \|^2_{\U^3} \gg \| A \|^2_{\U^3} \ge_M \E^2_3 (A) \,,
$$
$$
    \E_4 (A') \E(A') \ge \| A' \|^2_{\U^3} \gg \| A \|^2_{\U^3} \ge_M \E_4 (A) \E(A) \ge \E_4 (A) \E(A')
$$
    and, hence, $\E_3 (A') \sim_M \E_3 (A)$, $\E_4 (A') \sim_M \E_4 (A)$.
    With some abuse of the notation we will use the same letter $A$ for $A'$ below.
    By Lemma \ref{l:E_3_A_s}, we have
\begin{equation}\label{f:23.03.2014_1}
    \| A \|_{\U^3} = \sum_s \E (A_s) = \sum_s \sum_t (A_s \c A_s)^2 (t)
        \gg_{M} \sum_s |A_s|^3 = \sum_s \E(A,A_s) = \E_3 \,.
\end{equation}
    One can assume that the summation in the last formula is taken over $s$ such that
    $|A_s| \gg_M \E_3 \E^{-1}$
    and $\E (A_{s}) \gg_M |A_{s}|^3$, $|A_{s}|^3 \gg_M \E(A,A_s)$.
    Let us consider the condition $\E (A_{s}) \gg_M |A_{s}|^3$.
    By Balog--Szemer\'{e}di--Gowers Theorem we can find $H^s \subseteq A_{s}$ with
    $|H^s| \gg_M |A_{s}|$, and $|H^s - H^s| \ll_M |H^s|$.
    Loosing a logarithm $l=\log a$ we can assume that the summation in (\ref{f:23.03.2014_1}) is taken over
    $|A_s|$, $\D < |A_s| \le 2 \D$,
    $\Delta \gg_{M,\,l} \E_3 \E^{-1}$ and $\E (A_{s}) \gg_{M,\,l} |A_{s}|^3$, $|A_{s}|^3 \gg_{M,\,l} \E(A,A_s)$.
    By $P$ denote the set of such $s$.
    Thus, $|P| \D^3 \gg_{M,\,l} \E_3$ and
    it is easy to check that
    $P=-P$.
    Note also that $\E(A,H^s) \ll_{M,\,l} |H^s|^3$ for any $s\in P$.
    We have
    \begin{equation}\label{f:27.04.2014_1}
        \sum_{s,t\in P} (A_s \c A_s)^2 (t) \gg_M \E_3 \,.
    \end{equation}
    Returning to (\ref{f:E_4_begin}), we obtain
    $$
        (|P| \D^3)^2 \gg_{M,\,l} \max \{ |P| \D^4 \E, \E_4 |P| \D^2  \}
    $$
    and hence $|P| \D^2 \gg_{M,\,l} \E$, $|P| \D^4 \gg_{M,\,l} \E_4$.
    Thus, $\D \sim_{M,\,l} \E_3 \E^{-1}_2$ and
    $$
        \E_4 \sim_{M,\,l} \D \E_3 \sim_{M,\,l} \D^2 \E \sim_{M,\,l} \D^4 |P| \,.
    $$
    So, we know all energies $\E$, $\E_3$, $\E_4$ if we know $|P|$ and $\D$.
    Let us estimate
    the size of the $P$.
    Taking any $s\in P$, we get by Lemma \ref{l:eigen_A'} that
    $$
        \D^3 \ll_{M,\,l} \E(A_s) \le \E(A,A_s) \ll \E a^{-1} \D \ll_{M,\,l} |P| \D^3 a^{-1}
    $$
    or
    $|P| \gg_M a$.
    So, we have proved (\ref{f:self-dual0})---(\ref{f:self-dual1}).

    Further
    $$
        \sum_{s,t\in P} |H^s \cap H^t| \le \sum_{s,t\in P} |A_s \cap A_t|  := \sigma \,.
    $$
    Applying Lemma \ref{l:disjoint_probab}, we find disjoint sets
    $H_j \subseteq A_{s_j}$, $|H_j| \gg_M |A_{s_j}|$, $|H_j-H_j| \le |H_j-H_j| \ll_{M,\,l} |H_j|$,
    $|H_j| \gg \D \sim_{M,\,l} \E_3(A) \E(A)^{-1}$,
    $j\in [k]$ and $k\gg |P|^2 \D \sigma^{-1}$.
    Arguing as in Corollary \ref{c:disjoint_P}, we get $\sigma \ll \E |P| |A|^{-1}$
    and hence $k \gg |P| \D |A| \E^{-1} \gg_{M,\,l} |A| \D^{-1}$.
    Of course the
    last
    bound on $k$ is the best possible up to constants depending on $M$, $l$.
    We have obtained the necessary condition.

    Let us prove the sufficient condition.
    Using the Cauchy--Schwarz inequality and formulas (\ref{f:self-dual0})---(\ref{f:self-dual1}), we have
$$
    \| A \|^2_{\U^3} \ge \left( \sum_{s\in P} \E(A_s) \right)^2 \ge
        \left( \sum_{s\in P} \E(H_s) \right)^2
            \gg_{M,\,l}
            \left( \sum_{s\in P} |H_s|^3 \right)^2
            \gg_{M,\,l}
$$
$$
            \gg_{M,\,l}
                       |P|^2 \D^6
                            \gg_{M,\,l}
                                \E(A) \E_4 (A)
$$
as required.
This completes the proof.
$\hfill\Box$
\end{proof}

\bigskip

\begin{remark}
    In the statement of Theorem \ref{t:self-dual}
    there is
    the set of popular differences $P$
    and the structure of $A$ is described in terms of the set $P$.
    Although, we have obtained a criterium it can be named as a weak structural result.
    Perhaps, a stronger version avoiding using of the set $P$ takes place, namely,
    under the hypothesis  of
    Theorem \ref{t:self-dual} there are disjoint
    sets $H_j \subseteq A_{s_j}$, $|H_j| \gg_{M,\,l} \D$,
    $\D \sim_{M,\,l} \E_3(A) \E(A)^{-1}$,  $j\in [k]$
    such that (\ref{f:self-dual1})
    holds
    and
    \begin{equation}\label{f:self-dual2_conj}
    \sum_{j=1}^k |H_j|^4 \gg_{M,\,l} \E_3 (A) \,,
    \quad \quad \sum_{j=1}^k |H_j|^3 \gg_{M,\,l} \E(A) \,,
    \quad \quad \sum_{j=1}^k |H_j|^5 \gg_{M,\,l} \E_4(A)\,.
\end{equation}
    It is easy to see that it is a sufficient condition.
    Indeed, because of the sets $H_j \subseteq A$ are disjoint, we have
$$
    \| A \|_{\U^3} \ge \sum_{j=1}^k \| H_j \|_{\U^3} \,.
$$
Using the assumption $|H_j-H_j| \ll_{M,\,l} |H_j|$,
the first bound from (\ref{f:self-dual2_conj}), as well as Corollary \ref{c:U^3&doubling}, we obtain
$$
    \| A \|_{\U^3} \gg_{M,\,l} \sum_{j=1}^k |H_j|^4 \gg_{M,\,l} \E_3 (A)
$$
and, similarly, by the second and the third inequality of (\ref{f:self-dual2_conj}), we get
$$
    \| A \|^2_{\U^3} \gg_{M,\,l} \E (A) \E_4 (A)
$$
as required.
\end{remark}

\bigskip

\begin{example}
    Let $A\subseteq \Gr$ be a set having small Wiener norm, that is
    the following quantity $\| A \|_W := N^{-1} \sum_{\xi} |\FF{A} (\xi)| := M$
    is small.
    Then for any $B\subseteq A$, applying the Parseval identity,  one has
$$
    |B| = \sum_{x} B(x) A(x) = N^{-1} \sum_{\xi} \FF{B} (\xi) \ov{\FF{A} (\xi)} \,.
$$
    Using the H\"{o}lder inequality twice (see also \cite{KSh}), we get
$$
    |B|^4 \le \left( N^{-1} M \sum_{\xi} |\FF{B} (\xi)|^2 |\FF{A} (\xi)| \right)^2
        \le
            M^2 |B| \E(A,B)
$$
or, in other words,
\begin{equation}\label{f:E(A,B)_Wiener}
    \E(A,B) \ge \frac{|B|^3}{M^2} \,.
\end{equation}
    By multiplicative property of Wiener norm, we have $\| A_s \|_W \le M^2$.
    In particular, $\E(A_s) \ge \frac{|A_s|^3}{M^4}$.
    Hence $\| A \|_{\U^3} \ge M^{-4} \E_3 (A)$.
    Further, $\E(A) \ge M^{-2} |A|^3$, $\E_3 (A) \ge M^{-4} |A|^4$
    and hence
    $$
        \| A \|^2_{\U^3} \ge M^{-8} \E^2_3 (A) \ge M^{-16} \E(A) \E_4 (A) \,.
    $$
    Thus, an application of Theorem \ref{t:self-dual}
    gives us
    that $A$ has very explicit structure
    (2--connectedness follows from (\ref{f:E(A,B)_Wiener}) and $U^3$--connectedness can be obtained
    via formula (\ref{f:U^3(A,B)}) in a similar way).
    Another structural result on sets from $\F_p$ with small Wiener norm was given in \cite{KSh}.
\end{example}

\bigskip

If Theorem \ref{t:E(A_s)} does not hold that is $\E(A_s) \gg |A_s|^3$ for all $s$ then, clearly,
$\| A \|_{\U^3} \gg \E_3 (A)$
and we can try to apply our structural Theorem \ref{t:self-dual}.
On the other hand, if $A$ is a self--dual set, that is a disjoint union of sets with small doubling
then for any $s\neq 0$ one has exactly $\E(A_s) \gg |A_s|^3$.
It does not contradict to Theorem \ref{t:E(A_s)} because of condition (\ref{cond:A_s_bounded}).

\bigskip

Roughly speaking, in the proof of Theorem \ref{t:self-dual} we found disjoint subsets of $A_s$,
containing huge amount of
the
energy (see also Corollary \ref{c:disjoint_P}).
One can ask about the possibility to find some number of disjoint $A_s$ (and not its subsets)
in general situation.
Our next statement answer the question affirmatively.

\begin{proposition}
    Let $A\subseteq \Gr$ be a set, $D\subseteq A-A$.
    Put
\begin{equation}\label{cond:A^2-D(A)}
    \sigma := \sum_{s \in D} |A-A_s| \,.
\end{equation}
    Then there are at least $l \ge |D|^2/(4\sigma)$ disjoint sets $A_{s_1}, \dots, A_{s_l}$.
    In particular, if
\begin{equation}\label{cond:A^2-D(A)'}
        |A^2 - \Delta (A)| \le \frac{|A-A|^2}{M}
\end{equation}
    then there are at least $l \ge M/4$ disjoint sets $A_{s_1}, \dots, A_{s_l}$.
\label{p:disjoint_via_A-A_s}
\end{proposition}
\begin{proof}
    Our arguments is a sort of an algorithm.
    By (\ref{cond:A^2-D(A)}) there is $s_1 \in D$ such that $|A-A_{s_1}| \le \sigma/|D|$.
    Put $D_1 = D\setminus (A-A_{s_1})$.
    If $|D_1| < |D|/2$ then terminate the algorithm.
    If not then by an obvious estimate
    $$
        \sum_{s\in D_1} |A-A_s| \le \sigma
    $$
    we find $s_2 \in D_1$ such that
    $$
        |A-A_{s_2}| \le \frac{\sigma}{|D_1|} \le \frac{2\sigma}{|D|} \,.
    $$
    Put $D_2 = D_1 \setminus (A-A_{s_2})$.
    If $|D_1| < |D|/2$ then terminate the algorithm.
    And so on.
    At the last step, we obtain the set $D_l = D\setminus \bigcup_{j=1}^l (A-A_{s_j})$, $|D_l|<|D|/2$.
    It follows that
    $$
        \frac{|D|}{2} \le |\bigcup_{j=1}^l (A-A_{s_j})| \le \sum_{j=1}^l |A-A_{s_j}| \le l \frac{2\sigma}{|D|} \,.
    $$
    Thus $l \ge |D|^2/(4\sigma)$.
    Finally, recall that
\begin{equation}\label{f:A-A_s}
    t\in A-A_s \quad \mbox{ iff } \quad A_t \cap A_s \neq \emptyset \,.
\end{equation}
    Thus all constructed sets  $A_{s_1}, \dots, A_{s_l}$ are disjoint.

    To get (\ref{cond:A^2-D(A)'}) put $D=A-A$ and recall that by Lemma \ref{l:A^2_pm}
    the following holds $|A^2 - \D(A)| = \sum_{s\in A-A} |A-A_s|$.
    This completes the proof.
$\hfill\Box$
\end{proof}

\bigskip

One can ask is it true that not only $\E_3$ energy but $U^3$--norm of sumsets or difference sets is large?
It is easy to see that the answer is no, because of our basic example $A=H\dotplus \L$, $|\L| = K$.
In the case $\E(A) \sim |A|^3/K$, $\E_3 (A) \sim |A|^4 /K$ but $\| A\|_{\U^3} \sim |A|^4 /K^2$
and similar for $A\pm A$.

\section{Appendix}
\label{sec:appendix}

In the section we prove that any set contains a relatively large connected subset.
The case $k=2$ of Proposition \ref{p:connected_k} below was proved in \cite{s_doubling}
(with slightly worse constants)
and we begin with
a wide generalization.

\begin{definition}
    Let $X$, $Y$ be two nonempty sets, $|X|=|Y|$.
    A nonnegative symmetric function $q(x,y)$, $x\in X$, $y\in Y$
    is called {\bf weight} if the correspondent
    matrix $q(x,y)$ is nonnegatively defined.
\end{definition}
Having two sets $A$ and $B$ put $\E_q (A,B) := \sum_{x,y} q(x,y) A(x) B(y)$, $\E_q (A) := \E_q (A,A)$.
Clearly, $\E_q (A,B) \le  |A| |B| \| q \|_\infty$.
The main property of any weight is the following.

\begin{lemma}
Let $q$ be  a weight.
Then for any sets $A,B$, one has
$$
    \E^2_q (A,B) \le \E_q (A) \E_q (B) \,.
$$
\label{l:weight_property}
\end{lemma}

\begin{example}
    Clearly, the function $q(x,y)=(B\c B)^k (x-y)$ for any set $B$
    and an arbitrary positive integer $k$
    is a weight.
    Further, by the construction of Gowers $U^d$--norms it follows that
\begin{equation}\label{f:q_Gowers}
    q_d (x_1,x'_1) = \sum_{x_2,\dots,x_d \in \Gr}\, \sum_{x'_2,\dots,x'_d \in \Gr}\,
                    \prod_{\o \in \{ 0,1 \}^d} f ({\rm pr} (\v{x}^{\o}))
\end{equation}
is also a weight for any nonnegative function $f$.
    In formula (\ref{f:q_Gowers}), we have $\v{x} = (x_1,\dots,x_d)$, $\v{x}' = (x'_1,\dots,x'_d)$,
    ${\rm pr} (y_1,\dots,y_d) := y_1 + \dots + y_d$.
    Another example of a weight is
 \begin{equation}\label{f:q_Gowers+}
    q^*_d (x,y)
        =
            \sum_{h_1,\dots,h_{d-1}}\, \prod_{\o \in \{ 0,1 \}^{d-1},\, \o \neq 0}
                f(x + \o\cdot \v{h}) f(y + \o\cdot \v{h}) \,,
 \end{equation}
    where $f$ is an arbitrary nonnegative function again and $\v{h} = (h_1,\dots,h_{d-1})$.
\label{ex:weights}
\end{example}

For two sets $S,T \subseteq \Gr$, $S\neq \emptyset$, $T\subseteq S$ put $\mu_{S} (T) = |T|/|S|$.
Now we prove a general lemma on connected sets and quantities $\E_q$, where $q$ is a weight.

\begin{lemma}
    Let $A,B \subseteq \Gr$ be two sets, $\beta_1,\beta_2,\rho \in (0,1]$ be real numbers,
    $\beta_1 \le \beta_2$, $\rho < \beta_1 / \beta_2$.
    Let $q$ be a weight.
    Suppose that $\E_q (A) \ge c |A|^2 \| q \|_\infty$, $c\in (0,1]$.
    Then there is $A'\subseteq A$ such that for any subset $\t{A} \subseteq A'$,
    $\beta_1 |A'| \le |\t{A}| \le \beta_2 |A'|$ one has
\begin{equation}\label{f:connected_AB+}
    \E_q (\t{A}) \ge \rho^2 \mu^2_{A'} (\t{A}) \cdot \E_q (A')
    \,,
\end{equation}
    and besides
\begin{equation}\label{f:connected_AB_E+}
     \E_q (A') > (1-\beta_2 \rho)^{2s} \E_q (A) \,,
\end{equation}
    where $s\le \log (1/c) ( 2 \log (\frac{1-\beta_2 \rho}{1-\beta_1} ))^{-1}$.
\label{l:connected_AB+}
\end{lemma}
\begin{proof}
    Put $b=\| q \|_\infty$.
    We use an inductive procedure
    in the proof.
    Let us describe the first step of our algorithm.
    Suppose that (\ref{f:connected_AB+}) does not hold for some set $C \subseteq A$,
    $\beta_1 |A| \le |C| \le \beta_2 |A|$.
    Put $A^1 = A\setminus C$.
    Then $|A^1| \le (1-\beta_1) |A|$.
    Using
    Lemma \ref{l:weight_property}, we get
$$
    \E_q (A) = \E_q (C,A) + \E_q (A^1,A) < \rho \mu_A (C) \E_q (A) + \E^{1/2}_q (A^1) \E^{1/2}_q (A) \,.
$$
    Hence
$$
    \E_q (A^1) > \E_q (A) (1-\mu_A (C) \rho)^2 \ge \E_q (A) (1- \beta_2 \rho)^2 \,.
$$
    After that applying the same arguments to the set $A^1$, find a subset $C \subseteq A^1$ such that (\ref{f:connected_AB+}) does not hold (if it exists) and so on.
    We obtain a sequence of sets $A\supseteq A^1 \supseteq \dots \supseteq A^s$,
    and $|A^s| \le (1-\beta_1)^s |A|$.
    So, at the step $s$, we have
\begin{equation}\label{tmp:02.04.2014_1}
    c|A|^2 b (1-\beta_2 \rho)^{2s} \le \E_q (A) (1-\beta_2 \rho)^{2s} < \E_q (A^s) \le |A^s|^2 b
        \le
            (1-\beta_1)^{2s} |A|^2 b \,.
\end{equation}
    Thus, our algorithm must stop after at most $s\le \log (1/c) ( 2 \log (\frac{1-\beta_2 \rho}{1-\beta_1} ))^{-1}$ number of steps.
    Putting $A' = A^s$, we see that inequality (\ref{f:connected_AB+}) takes place for any $\t{A} \subseteq A'$ with
    $\beta_1 |A'| \le |\t{A}| \le \beta_2 |A'|$.
    Finally, by the second estimate in (\ref{tmp:02.04.2014_1}), we obtain (\ref{f:connected_AB_E+}).
This concludes the proof.
$\hfill\Box$
\end{proof}

\bigskip

Let us formulate a useful particular case of Lemma \ref{l:connected_AB+}.

\begin{lemma}
    Let $A,B \subseteq \Gr$ be two sets, $\beta_1,\beta_2,\rho \in (0,1]$ be real numbers,
    $\beta_1 \le \beta_2$, $\rho < \beta_1 / \beta_2$.
    Suppose that $\E(A,B) \ge c|A|^2 |B|$, $c\in (0,1]$.
    Then there is $A'\subseteq A$ such that for any subset $\t{A} \subseteq A'$,
    $\beta_1 |A'| \le |\t{A}| \le \beta_2 |A'|$ one has
\begin{equation}\label{f:connected_AB}
    \E (\t{A},B) \ge \rho^2 \mu^2_{A'} (\t{A}) \cdot \E(A',B)
    \,,
\end{equation}
    and besides
\begin{equation}\label{f:connected_AB_E}
     \E(A',B) > (1-\beta_2 \rho)^{2s} \E(A,B) \,,
\end{equation}
    where $s\le \log (1/c) ( 2 \log (\frac{1-\beta_2 \rho}{1-\beta_1} ))^{-1}$.
\label{l:connected_AB}
\end{lemma}

Lemma \ref{l:connected_AB} implies the required statement, generalizing the result from \cite{s_doubling}.

\begin{proposition}
    Let $A \subseteq \Gr$ be a set,  $\beta \in (0,1)$ be real numbers, and $k\ge 2$ be an integer.
    Put $c=\E_k(A) |A|^{-(k+1)}$.
    Then there is $A'\subseteq A$ such that
    \begin{equation}\label{p:connected_k_energy}
        \E_k (A',A) > (1- 2^{-1} \beta)^{2s} \E_k (A) \,,
    \end{equation}
    where $s\le \log (1/c) ( 2 \log (\frac{2-\beta}{2-2\beta} ))^{-1}$,
    and $A'$ is $(k,\beta,\gamma)$--connected with
    \begin{equation}\label{p:connected_k_gamma}
        \gamma \ge 2^{-(2sk+2k-2s)} \beta^{2k} (2-\beta)^{2s(k-1)} \,.
    \end{equation}
    In particular,
    \begin{equation}\label{p:connected_k_card}
        |A'| \ge (1-2^{-1} \beta)^s c^{1/2} |A| \,.
    \end{equation}
\label{p:connected_k}
\end{proposition}
\begin{proof}
Note
that $T\subseteq S$ iff $\D (T) \subseteq \D (S)$.
Applying Lemma \ref{l:connected_AB} with $A=\D(A)$, $B=A^{k-1}$, $\beta_1 = \beta$, $\beta_2 = 1$, $\rho = \beta_1/(2\beta_2) = \beta/2$,
and using formula (\ref{f:energy-B^k-Delta}), we find a set $A'\subseteq A$ such that
for any subset $\t{A} \subseteq A'$, $\beta |A'| \le |\t{A}|$ one has
\begin{equation}\label{tmp:connected_AB}
    \E_k (\t{A},A) \ge \rho^2 \mu^2_{A'} (\t{A}) \cdot \E_k (A',A)
    \,,
\end{equation}
and
\begin{equation}\label{tmp:connected_AB_E}
    \E_k (A',A) > (1- 2^{-1} \beta)^{2s} \E_k (A) \,,
\end{equation}
where $s\le \log (1/c) ( 2 \log (\frac{1-\rho}{1-\beta} ))^{-1}$.
We have obtained inequality (\ref{p:connected_k_energy}).
From (\ref{tmp:connected_AB}), (\ref{tmp:connected_AB_E}) and the H\"{o}lder inequality, we get
$$
    \E_k (\t{A}) \ge \rho^{2k} \mu^{2k}_{A'} (\t{A}) \E^k_k (A',A) \E^{-(k-1)}_k (A)
        \ge
            (2^{-1} \beta)^{2k} (1- 2^{-1} \beta)^{2s(k-1)} \mu^{2k}_{A'} (\t{A}) \E_k (A',A)
                \ge
$$
$$
                \ge
                    (2^{-1} \beta)^{2k} (1- 2^{-1} \beta)^{2s(k-1)} \mu^{2k}_{A'} (\t{A}) \E_k (A')
                        \,.
$$
Thus, the set $A'$ is $(k,\beta,\gamma)$--connected with $\gamma$ satisfying (\ref{p:connected_k_gamma}).
To obtain (\ref{p:connected_k_card}) just apply (\ref{tmp:connected_AB_E}) and a trivial upper bound
for $\E_k (A',A)$
$$
    |A'|^2 |A|^{k-1} \ge \E_k (A',A) > (1- 2^{-1} \beta)^{2s} \E_k (A) = (1- 2^{-1} \beta)^{2s} c|A|^{k+1}
$$
as required.
This completes the proof.
$\hfill\Box$
\end{proof}

    In view of Lemma \ref{l:connected_AB+} and Example \ref{ex:weights} one can obtain an analog of Proposition \ref{p:connected_k} for $U^k (\beta,\gamma)$--connected sets, see Definition \ref{def:Uk-conn}.
    We leave the details to an interested reader.

\bigskip

\noindent{I.D.~Shkredov\\
Steklov Mathematical Institute,\\
ul. Gubkina, 8, Moscow, Russia, 119991}
\\
and
\\
IITP RAS,  \\
Bolshoy Karetny per. 19, Moscow, Russia, 127994\\
{\tt ilya.shkredov@gmail.com}


\begin{thebibliography}{99}


\bibitem{ALON} N. Alon, V. R\"{o}dl,
\emph{Sharp bounds for some multicolor Ramsey numbers, } Combinatorica, 25 (2005), 125--141.


\bibitem{Austin} T. Austin,
\emph{On the Norm Convergence of Nonconventional Ergodic Averages, }
preprint.


\bibitem{BK_AP3} M.~Bateman, N.~Katz,
\emph{New bounds on cap sets, }
arXiv:1101.5851v1 [math.CA] 31 Jan 2011.


\bibitem{BK_struct} M.~Bateman, N.~Katz,
\emph{Structure in additively nonsmoothing sets, }
arXiv:1104.2862v1 [math.CO] 14 Apr 2011.


\bibitem{Bol_Th} B.~Bollob\'{a}s, A.~Thomason,
\emph{Projections of bodies and hereditary properties of hypergraphs, }
Bull. London Math. Soc. 27 (1995) 417--424.



    \bibitem{Fu} H. Furstenberg,
    \emph{Recurrence in ergodic theory and combinatorial number theory,  }
    Princeton N.J., 1981.


    \bibitem{GreenA+A} B. Green,
    \emph{Arithmetic Progressions in Sumsets, }
    Geom. Funct. Anal., \textbf{12} (2002) no. 3, 584--597.


       \bibitem{GT_great} B. Green, T. Tao,
    \emph{The primes contain arbitrarily long arithmetic progressions, }
    Annals of Math. (2), {\bf 167}:2 (2008), 481--547.


    \bibitem{GT_U3} B. Green, T. Tao,
    \emph{An inverse theorem for the Gowers $U^3(G)$ norm, }
    Proc. Edinb. Math. Soc. (2), {\bf 51}:1 (2008), 73--153.


    \bibitem{GT_polynomials} B. Green, T. Tao,
    \emph{The quantitative behaviour of polynomial orbits on nilmanifolds, }
    Ann. of Math. (2), {\bf 175}:2 (2012), 465--540.


    \bibitem{GT_equiv} B. Green, T. Tao,
    \emph{The equivalence between inverse sumset theorems and inverse conjectures for the $U^3$---norm, }
    Math. Proc. Cambridge Philos. Soc., {\bf 149}:1 (2010), 1--19.


    \bibitem{GT_Mobius} B. Green, T. Tao,
    \emph{Quadratic uniformity for the M\"{o}bius function, }
    Ann. Inst. Fourier (Grenoble), {\bf 58}:6 (2008), 1863--1935.


    \bibitem{GT_nil} B. Green, T. Tao,
    \emph{The quantitative behaviour of polynomial orbits on nilmanifolds, }
    arXiv:0709.3562v1 [math.NT] 22 Sep 2007.


    \bibitem{Gow_4} W.T. Gowers,
    \emph{A new proof of Szemer\'{e}di's theorem for arithmetic progressions of length four, }
    Geom. func. anal., v.8, 1998, 529--551.


    \bibitem{Gow_m} W.T. Gowers,
    \emph{A new proof of Szemer\'{e}di's theorem, }
    Geom. Funct. Anal. \textbf{11} (2001), 465--588.





\bibitem{HK1} B. Host, B. Kra,
\emph{Nonconventional ergodic averages and nilmanifolds, }
Ann. of Math. (2), {\bf 161}:1 (2005), 397--488.


\bibitem{HK2} B. Host, B. Kra,
\emph{Convergence of polynomial ergodic averages, }
Israel J. Math., {\bf 149}:1--19 (2005).


\bibitem{kk} N. H. Katz and P. Koester,
\emph{On additive doubling and energy, }
SIAM J. Discrete Math. {\bf 24} (2010), 1684--1693.


\bibitem{KSh} S.V. Konyagin, I.D. Shkredov,
\emph{On Wiener norm of subsets of $\Z_p$ of medium size, }
arXiv:1403.8129v1 [math.CA].


\bibitem{Samorod_false} S. Lovett, R. Meshulam, A. Samorodnitsky,
\emph{Inverse Conjecture for the Gowers norm is false, }
Theory Comput., {\bf 7} (2011), 131--145.


\bibitem{M_R-N_S} B. Murphy, O. Roche-Newton and I.D. Shkredov,
\emph{Variations on the sum-product problem, }
arXiv:1312.6438v2 [math.CO].


\bibitem{petridis} G.~Petridis,
\emph{New Proofs of Pl\"{u}nnecke--type Estimates for Product Sets in Non-Abelian Groups, } preprint.


\bibitem{Rudin_book} W.~Rudin,
\emph{Fourier analysis on groups, }  Wiley 1990 (reprint of the 1962 original).


    \bibitem{Samorod} A. Samorodnitsky, L. Trevisan,
    \emph{Gowers uniformity, influence of variables, and PCPs, }
    SIAM J. Comput., {\bf 39}:1 (2009), 323--360.




\bibitem{Sanders_2A-2A} T. Sanders,
\emph{On the Bogolyubov-Ruzsa lemma, } Anal. PDE, to appear, 2010, arXiv:1011.0107.


\bibitem{Sanders_survey2} T. Sanders,
\emph{Approximate (abelian) groups, }
arXiv:1212.0456 [math.CA].


\bibitem{schoen_BSzG} T. Schoen,
\emph{New bounds in Balog--Szemer\'{e}di--Gowers theorem, } preprint.


\bibitem{SS1} T. Schoen and I. Shkredov,
\emph{Higher moments of convolutions, }
J. Number Theory 133 (2013), no. 5, 1693--1737.


\bibitem{SS2} T. Schoen and I. Shkredov,
\emph{Additive properties of multiplicative subgroups of $\mathbb{F}_p$, }
Q. J. Math. 63 (2012), no. 3, 713--722.


\bibitem{SS3} T. Schoen and I. Shkredov,
\emph{On sumsets of convex sets, }
Combin. Probab. Comput. 20 (2011), no. 5, 793--798.




\bibitem{s_doubling} I. Shkredov,
\emph{On Sets with Small Doubling, } Mat. Zametki, {\bf 84}:6 (2008), 927--947.


\bibitem{s_ineq} I. Shkredov,
\emph{Some new inequalities in additive combinatorics, }
MJCNT, {\bf 3}:2 (2013), 237--288.


\bibitem{s_mixed} I. Shkredov,
\emph{Some new results on higher energies, }
Transactions of MMS, 74:1 (2013), 35--73.


\bibitem{SV} I. Shkredov and I. Vyugin,
\emph{On additive shifts of multiplicative subgroups, }
Sb. Math. 203 (2012), no. 5-6, 844--863.


\bibitem{Tao1} T. Tao,
\emph{A quantitative ergodic theory proof of Szemer\'{e}di's theorem, }
Electron. J. Combin., {\bf 13}:1 (2006), Research Paper 99.


\bibitem{Tao2} T. Tao,
\emph{Norm convergence of multiple ergodic averages for commuting transformations, }
Ergodic Theory Dynam. Systems, {\bf 28}:2 (2008), 657--688.


\bibitem{TV} T. Tao and V. Vu,
\emph{Additive Combinatorics, }
Cambridge University Press (2006).


    \bibitem{TZ} T. Tao, T. Ziegler,
    \emph{The inverse conjecture for the Gowers norm over finite fields via the correspondence principle},
    Anal. PDE, {\bf 3}:1 (2010), 1--20.


\bibitem{Z1} T. Ziegler,
\emph{A non-conventional ergodic theorem for a nilsystem, }
Ergodic Theory Dynam. Systems, {\bf 25}:4 (2005), 1357--1370.


\bibitem{Z2} T. Ziegler,
\emph{Universal characteristic factors and Furstenberg averages, }
J. Amer. Math. Soc., {\bf 20}:1 (2007), 53--97.


\end{thebibliography}
\end{document}